\documentclass[10pt,reqno,oneside]{amsproc}
\title[Stabilization by a vacuum bubble]{Stabilization  of an incompressible fluid-elastic structure system using a vacuum bubble}

\author[B.~Ingimarson]{Benjamin Ingimarson}
\address{Department of Mathematics, University of Southern California, Los Angeles, CA 90089}
\email{ingimars@usc.edu}

\author[I.~Kukavica]{Igor Kukavica}
\address{Department of Mathematics, University of Southern California, Los Angeles, CA 90089}
\email{kukavica@usc.edu}

\author[Wojciech~S.~O\.{z}a\'{n}ski]{Wojciech~S.~O\.{z}a\'{n}ski}
\address{Department of Mathematics, Florida State University, Tallahassee, FL 32306, and Department of Mathematics, Princeton University, Princeton, NJ 08540}
\email{wozanski@fsu.edu}
  \chardef\forshowkeys=0
  \chardef\refcheck=0
  \chardef\showllabel=0
  \chardef\sketches=0
  \chardef\figures=1
\usepackage{enumitem,mathtools}
\usepackage{datetime}
\usepackage{fancyhdr}
\usepackage{comment}
\allowdisplaybreaks
\ifnum\forshowkeys=1

  \usepackage[notref,notcite,color]{showkeys}
\fi
\usepackage[margin=1in]{geometry}
\usepackage{amsmath, amsthm, amssymb}
\usepackage{times}
\usepackage{graphicx}
\usepackage[usenames,dvipsnames,svgnames,table]{xcolor}
\usepackage{marginnote}
\usepackage[unicode,breaklinks=true,colorlinks=true,linkcolor=blue,urlcolor=blue,citecolor=blue]{hyperref}
\usepackage[most]{tcolorbox}
\usepackage{tikz}

\ifnum\refcheck=1
  \usepackage{refcheck}
\fi

\usepackage{import}
\usepackage{xifthen}
\usepackage{pdfpages}
\usepackage{transparent}

\begin{document}
\def\YY{X}
\def\OO{\mathcal O}
\def\SS{\mathbb S}
\def\CC{\mathbb C}
\def\RR{\mathbb R}
\def\TT{\mathbb T}
\def\ZZ{\mathbb Z}
\def\HH{\mathbb H}
\def\RSZ{\mathcal R}
\def\LL{\mathcal L}
\def\SL{\LL^1}
\def\ZL{\LL^\infty}
\def\GG{\mathcal G}
\def\tt{\langle t\rangle}
\def\erf{\mathrm{Erf}}
\def\mgt#1{\textcolor{magenta}{#1}}
\def\ff{\rho}
\def\gg{G}
\def\sqrtnu{\sqrt{\nu}}
\def\ww{w}
\def\ft#1{#1_\xi}
\def\les{\lesssim}
\def\lec{\lesssim}
\def\ges{\gtrsim}
\renewcommand*{\Re}{\ensuremath{\mathrm{{\mathbb R}e\,}}}
\renewcommand*{\Im}{\ensuremath{\mathrm{{\mathbb I}m\,}}}
\ifnum\showllabel=1
 \def\llabel#1{\marginnote{\color{lightgray}\rm\small(#1)}[-0.0cm]\notag}
\else
 \def\llabel#1{\notag}
\fi
\newcommand{\norm}[1]{\left\|#1\right\|}
\newcommand{\nnorm}[1]{\lVert #1\rVert}
\newcommand{\abs}[1]{\left|#1\right|}
\newcommand{\NORM}[1]{|\!|\!| #1|\!|\!|}
\newtheorem{theorem}{Theorem}[section]
\newtheorem{Theorem}{Theorem}[section]
\newtheorem{corollary}[theorem]{Corollary}
\newtheorem{Corollary}[theorem]{Corollary}
\newtheorem{proposition}[theorem]{Proposition}
\newtheorem{Proposition}[theorem]{Proposition}
\newtheorem{Lemma}[theorem]{Lemma}
\newtheorem{lemma}[theorem]{Lemma}
\theoremstyle{definition}
\newtheorem{definition}{Definition}[section]
\newtheorem{remark}[Theorem]{Remark}
\def\theequation{\thesection.\arabic{equation}}
\numberwithin{equation}{section}
\definecolor{mygray}{rgb}{.6,.6,.6}
\definecolor{myblue}{rgb}{9, 0, 1}
\definecolor{colorforkeys}{rgb}{1.0,0.0,0.0}
\newlength\mytemplen
\newsavebox\mytempbox
\makeatletter
\newcommand\mybluebox{%
    \@ifnextchar[
       {\@mybluebox}%
       {\@mybluebox[0pt]}}
\def\@mybluebox[#1]{%
    \@ifnextchar[
       {\@@mybluebox[#1]}%
       {\@@mybluebox[#1][0pt]}}
\def\@@mybluebox[#1][#2]#3{
    \sbox\mytempbox{#3}%
    \mytemplen\ht\mytempbox
    \advance\mytemplen #1\relax
    \ht\mytempbox\mytemplen
    \mytemplen\dp\mytempbox
    \advance\mytemplen #2\relax
    \dp\mytempbox\mytemplen
    \colorbox{myblue}{\hspace{1em}\usebox{\mytempbox}\hspace{1em}}}
\makeatother
\def\rr{r}
\def\weaks{\text{\,\,\,\,\,\,weakly-* in }}
\def\inn{\text{\,\,\,\,\,\,in }}
\def\cof{\mathop{\rm cof\,}\nolimits}
\def\Dn{\frac{\partial}{\partial N}}
\def\Dnn#1{\frac{\partial #1}{\partial N}}
\def\tdb{\tilde{b}}
\def\tda{b}
\def\qqq{u}
\def\lat{\Delta_2}
\def\biglinem{\vskip0.5truecm\par==========================\par\vskip0.5truecm}
\def\inon#1{\hbox{\ \ \ \ \ \ \ }\hbox{#1}}                
\def\onon#1{\inon{on~$#1$}}
\def\inin#1{\inon{in~$#1$}}
\def\FF{F}
\def\andand{\text{\indeq and\indeq}}
\def\ww{w(y)}
\def\ll{{\color{red}\ell}}
\def\ee{\mathrm{e}}
\def\startnewsection#1#2{ \section{#1}\label{#2}\setcounter{equation}{0}}   
\def\nnewpage{ }
\def\sgn{\mathop{\rm sgn\,}\nolimits}    
\def\Tr{\mathop{\rm Tr}\nolimits}    
\def\div{\mathop{\rm div}\nolimits}
\def\curl{\mathop{\rm curl}\nolimits}
\def\dist{\mathop{\rm dist}\nolimits}
\def\id{\mathop{\rm id}\nolimits}
\def\supp{\mathop{\rm supp}\nolimits}
\def\indeq{\quad{}}           
\def\period{.}                       
\def\semicolon{\,;}                  
\def\nts#1{{\cor #1\cob}}
\def\cmi#1{\text{~{{\coli IK: \underline{#1}}}~}}
\def\coli{\color{colorigor}}
\definecolor{colorigor}{rgb}{1, 0.2, 0.8}
\def\colr{\color{red}}
\def\colrr{\color{black}}
\def\colb{\color{black}}
\def\coly{\color{lightgray}}
\definecolor{colorgggg}{rgb}{0.1,0.5,0.3}
\definecolor{colorllll}{rgb}{0.0,0.7,0.0}
\definecolor{colorhhhh}{rgb}{0.3,0.75,0.4}
\definecolor{colorpppp}{rgb}{0.7,0.0,0.2}
\definecolor{coloroooo}{rgb}{0.45,0.0,0.0}
\definecolor{colorqqqq}{rgb}{0.1,0.7,0}
\def\colg{\color{colorgggg}}
\def\collg{\color{colorllll}}
\def\cole{\color{coloroooo}}
\def\coleo{\color{colorpppp}}
\def\cole{\color{black}}
\def\colu{\color{blue}}
\def\colc{\color{colorhhhh}}
\def\colW{\colb}   
\definecolor{coloraaaa}{rgb}{0.6,0.6,0.6}
\def\colw{\color{coloraaaa}}
\def\comma{ {\rm ,\qquad{}} }            
\def\commaone{ {\rm ,\quad{}} }          
\def\les{\lesssim}
\def\nts#1{{\color{blue}\hbox{\bf ~#1~}}} 
\def\ntsf#1{\footnote{\color{colorgggg}\hbox{#1}}} 
\def\blackdot{{\color{red}{\hskip-.0truecm\rule[-1mm]{4mm}{4mm}\hskip.2truecm}}\hskip-.3truecm}
\def\bluedot{{\color{blue}{\hskip-.0truecm\rule[-1mm]{4mm}{4mm}\hskip.2truecm}}\hskip-.3truecm}
\def\purpledot{{\color{colorpppp}{\hskip-.0truecm\rule[-1mm]{4mm}{4mm}\hskip.2truecm}}\hskip-.3truecm}
\def\greendot{{\color{colorgggg}{\hskip-.0truecm\rule[-1mm]{4mm}{4mm}\hskip.2truecm}}\hskip-.3truecm}
\def\cyandot{{\color{cyan}{\hskip-.0truecm\rule[-1mm]{4mm}{4mm}\hskip.2truecm}}\hskip-.3truecm}
\def\reddot{{\color{red}{\hskip-.0truecm\rule[-1mm]{4mm}{4mm}\hskip.2truecm}}\hskip-.3truecm}
\def\tdot{{\color{green}{\hskip-.0truecm\rule[-.5mm]{3mm}{3mm}\hskip.2truecm}}\hskip-.1truecm}
\def\gdot{\greendot}
\def\bdot{\bluedot}
\def\ydot{\cyandot}
\def\rdot{\cyandot}
\def\fractext#1#2{{#1}/{#2}}
\def\ii{\hat\imath}
\def\fei#1{\textcolor{blue}{#1}}
\def\vlad#1{\textcolor{cyan}{#1}}
\def\igor#1{\text{{\textcolor{colorqqqq}{#1}}}}
\def\igorf#1{\footnote{\text{{\textcolor{colorqqqq}{#1}}}}}
\def\Omf{\Omega_{\text f}}
\def\Ome{\Omega_{\text e}}
\def\Omb{\Omega_{\text b}}
\def\Gaf{\Gamma_{\text f}}
\def\Gae{\Gamma_{\text e}}
\def\Gab{\Gamma_{\text b}}
\def\Gac{\Gamma_{\text c}}
\def\Nf{N^{\text f}}
\def\Ne{N^{\text e}}

\newcommand{\p}{\partial}
\renewcommand{\d}{\mathrm{d}}
\newcommand{\UE}{U^{\rm E}}
\newcommand{\PE}{P^{\rm E}}
\newcommand{\KP}{K_{\rm P}}
\newcommand{\uNS}{u^{\rm NS}}
\newcommand{\vNS}{v^{\rm NS}}
\newcommand{\pNS}{p^{\rm NS}}
\newcommand{\omegaNS}{\omega^{\rm NS}}
\newcommand{\uE}{u^{\rm E}}
\newcommand{\vE}{v^{\rm E}}
\newcommand{\pE}{p^{\rm E}}
\newcommand{\omegaE}{\omega^{\rm E}}
\newcommand{\ua}{u_{\rm   a}}
\newcommand{\va}{v_{\rm   a}}
\newcommand{\omegaa}{\omega_{\rm   a}}
\newcommand{\ue}{u_{\rm   e}}
\newcommand{\ve}{v_{\rm   e}}
\newcommand{\omegae}{\omega_{\rm e}}
\newcommand{\omegaeic}{\omega_{{\rm e}0}}
\newcommand{\ueic}{u_{{\rm   e}0}}
\newcommand{\veic}{v_{{\rm   e}0}}
\newcommand{\up}{u^{\rm P}}
\newcommand{\vp}{v^{\rm P}}
\newcommand{\tup}{{\tilde u}^{\rm P}}
\newcommand{\bvp}{{\bar v}^{\rm P}}
\newcommand{\omegap}{\omega^{\rm P}}
\newcommand{\tomegap}{\tilde \omega^{\rm P}}
\newcommand{\eps}{\varepsilon}  
\newcommand{\eqnb}{\begin{equation}}
\newcommand{\eqne}{\end{equation}}
  
\renewcommand{\up}{u^{\rm P}}
\renewcommand{\vp}{v^{\rm P}}
\renewcommand{\omegap}{\Omega^{\rm P}}
\renewcommand{\tomegap}{\omega^{\rm P}}



\newcommand{\incfig}[2][\columnwidth]{%
  \def\svgwidth{#1}%
  \ifnum\figures=1
    \import{./figures/}{#2.pdf_tex}%
  \fi
}

\begin{abstract}
We prove a~priori estimates for the system of partial differential equations modeling the interaction between an elastic body and an incompressible fluid in a 3D curved domain.
The fluid is governed by the incompressible Navier-Stokes equations and contains a  bubble whose interior is a vacuum. 
The elastic body is described by a damped wave equation, and interaction with the fluid takes place along a free interface whose initial domain is curved. We show that the presence of the vacuum bubble stabilizes the system in the sense that it provides control of the average of the pressure function, and hence allows global existence and exponential decay of smooth solutions for small data.  
\end{abstract}

\maketitle
\setcounter{tocdepth}{2} 
\section{Introduction}\label{sec00}
We consider the model of a viscous incompressible fluid governed by the Navier-Stokes equations,
  \begin{align}
   u_t
   -\Delta u
   + (u \cdot \nabla) u
   + \nabla p
   &= 
   0
   ,
   \\
   \div u
   &=
   0
   ,
   \label{EQ0aa}
  \end{align}
interacting with an elastic body described by a damped wave equation,
  \begin{equation}
   w_{tt}
   -\Delta w
   + \alpha w_t
   =
   0
   .
   \label{EQ0ab}
  \end{equation}
The fluid domain includes a ``vacuum bubble'' whose interface with the fluid is modeled by the free-surface condition
  \begin{equation}
   (-pI + \nabla u)\nu
   =
   0
   ,
   \label{EQac}
  \end{equation}
where $\nu$ denotes the dynamic normal on the interface.
The initial domain is assumed to be curved with smooth boundaries; see Figure~\ref{F1}. 
In this paper, we provide a~priori estimates for small initial data and show that the energy of the system decays exponentially in time. Such a system, in the case of the flat domain $\TT^2 \times (0,2)$, was previously considered by \cite{KO}, who identified two quantities of the system,
\eqnb\label{two_suspects}
\int_{\Gac} w \quad \text{ and } \quad \int_{\Gac} q,
\eqne
whose exponential decay in time cannot be guaranteed. The authors of \cite{KO} consider an appropriate modification of the state variables and omit the lowest level energy estimate to bypass this problem, but the approach is strictly limited to the flat domain considered. In particular, global well-posedness with small data remains an open problem in the case of a general domain. \\

The main purpose of this paper is to show that, remarkably, introducing a mere vacuum bubble completely resolves the issues encountered by \cite{KO}, and applies to  any curved domain. In particular the vacuum bubble allows the lowest order energy estimate as well as allows us to control the decay of the quantities~\eqref{two_suspects}.
\\

This result lies within the broader context of fluid-structure interaction which has garnered substantial interest in literature both mathematical and physical.
The latter includes areas of applied sciences ranging from medicine to engineering, but in terms of purely mathematical interest, questions surrounding well-posedness of these systems have been given significant attention. 
Our system, in particular, relates to the dynamics of an elastic solid inside a viscous fluid. 
The principal difficulty surrounding these problems stems from the coupling of hyperbolic and parabolic dynamics, the former amounting to a loss of regularity. 

The first treatment of this problem was conducted by Coutand and Shkoller~\cite{CS1,CS2} who coupled the incompressible Navier-Stokes equations with the linear Kirchhoff equations and established local-in-time existence of smooth solutions; see~\cite{CS1}. 
Further results on local well-posedness for the incompressible case were obtained in~\cite{IKLT1,KT1,KT2}.
The case of compressible flows has also been examined, see~\cite{BG1,BG2,KT3}.
More relevant to our paper, global well-posedness for small data has been considered with extra stabilization conditions such as damping terms in the elastic equation.
This was first examined in~\cite{IKLT2}, where the authors included interior damping of the form $\alpha w_t + \beta w$ in the wave equation as well as damping along the common boundary of the fluid and elastic body.
In a later work~\cite{IKLT3}, the same authors considered the case with only interior damping.
Finally, global well-posedness for the case when $\beta = 0$ was studied in~\cite{KO} with a flat initial domain, periodic in the $(x_1,x_2)$ variables. 
The main conundrum of this setting is the lack of control on the lowest order terms for the pressure and the elastic displacement. 
Indeed, with no further stabilization conditions attached to the problem, lowest order control on either term requires control of the other, making the argument circular.
The authors bypassed this issue by controlling a ``double-normalized wave displacement'' instead of the usual elastic displacement $w$, which eliminates the need for lowest-order control of these terms.
The method does not  extend to the curved case due to the need for exclusion of $S=\id$
in the list of operators (see \eqref{EQ15} below) which in turn does not allow for absorption of lower order terms.

\section{The model}\label{sec01}

We consider a model of an elastic body $\Ome(t)$ interacting with a viscous incompressible fluid in $\Omf(t)$ where $\Gaf$ and $\Gae$ are rigid.
We assume all boundaries are smooth and define $\Ome(0) = \Ome$, $\Omf(0) = \Omf$ and similarly for $\Gab$ and~$\Gac$.
We have
  \begin{equation}
   \partial \Omf(t) 
   = \Gaf \cup \Gab(t) \cup \Gac(t)
   \label{EQ00a}
  \end{equation}
and 
  \begin{equation}
   \partial \Ome(t)
   = \Gac(t) \cup \Gae
   .
   \label{EQ00b}
  \end{equation}
See Figure~\ref{F1} for a sketch. 
  \begin{figure}[ht]
    \centering
    \incfig[0.4\columnwidth]{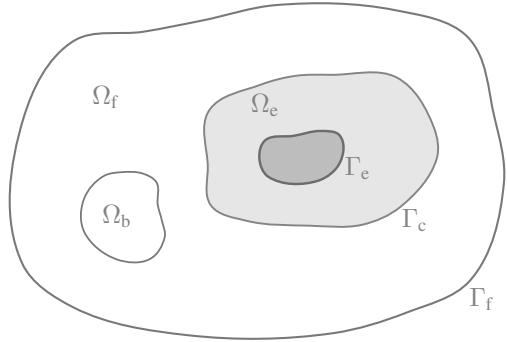}
    \caption{The sketch of the model at time $t=0$.}
    \label{F1}
  \end{figure}

\noindent
Note that $\Ome$ does not include the darkest shaded region inside the surface~$\Gae$.

The elastic body in $\Ome(t)$ is described by a displacement function $w(x,t) = \eta(x,t) - x$ where $\eta$ denotes the Lagrangian mapping of a particle at $x \in \Ome(0)$.
The displacement $w$ is determined by the equation,
  \begin{equation}
   w_{tt}
   - \Delta w
   + \alpha w_t
   = 0
   \inon{in $\Ome \times (0,\infty)$.}
   \label{EQ01}
  \end{equation}
We also assume zero displacement on $\Gae$,
  \begin{equation}
   w(x,t)
   = 0
   \inon{for $x \in \Gae$, $t > 0$}
   \label{EQ02}
   ,
  \end{equation}
representing $\Gae$ as the place where the elastic body is attached.
From the figure, we see that $\Gac(t)$ is shared by the elastic body and the fluid.
In $\Omf(t)$, the fluid is governed by the incompressible Navier-Stokes equations 
  \begin{align}
   \begin{split}
    u_t
    - \Delta u
    + (u \cdot \nabla)u
    + \nabla p 
    &= 0
    ,
    \\
    \div u 
    &= 0
    \inon{in $\Omf(t)$.}
   \end{split}
   \label{EQ03}
  \end{align}
  For the vacuum bubble $\Omb(t)$  we impose the assumption 
  \begin{equation}
   (-pI + \nabla u) \nu 
   = 0
   \label{EQ03a}
  \end{equation}
on the normal components of the stress at the free interface~$\Gab(t)$.
Here, we denote $\nu$ as the time-dependent normal vector on $\Gab(t)$, outward relative to the fluid domain~$\Omf(t)$.
Now, we reformulate the fluid equations into Lagrangian coordinates. 
In the fluid domain, the displacement $\eta(\cdot, t)$ is defined as a solution to 
  \begin{align}
   \begin{split}
    \eta_t(x,t)
    &= 
    v(x,t)
    ,
    \\
    \eta(0,x) 
    &=
    x
    ,
   \end{split}
   \label{EQ04}
  \end{align}
for $x \in \Omf$ and $t  > 0$, where $v$, representing the Lagrangian velocity, is defined as 
  \begin{equation}
   v(x,t) 
   = 
   u(\eta(x,t), t)
   .
   \label{EQ05}
  \end{equation}
We define the Lagrangian pressure $q(x,t) = p(\eta(x,t), t)$ as well.
Under the new coordinates, the Navier-Stokes equations take the form 
  \begin{align}
   \begin{split}
    \partial_t v_i
    - \partial_j (a_{j \ell} a_{k \ell} \partial_k v_i)
    + \partial_k (a_{k i} q) 
    &= 0
    \comma i = 1,2,3,
    \\
    a_{k i} \partial_k v_i 
    &= 0
   \end{split}
   \label{EQ06}
  \end{align}
in $\Omf \times (0,T)$, where $a = (\nabla \eta)^{-1}$.
Also note that by incompressibility of the fluid we have that $\det{a} = 1$. 
Hence, $a$ is its own cofactor matrix and the Piola identity
  \begin{equation}
   \partial_i a_{ij}
   = 0
   \label{EQ07}
  \end{equation}
holds.
We also assume the continuity of velocities on the common boundary,
  \begin{equation}
   w_t 
   =
   v
   \inon{ on $\Gac \times (0,T)$, }
   \label{EQ08}
  \end{equation}
and continuity of the normal components of the stresses between the elastic body and the fluid
  \begin{equation}
    - a_{j i}q N_j
    + a_{j \ell} a_{k \ell} \partial_k v_i N_j
    =
    \partial_j w_i N_j 
   \inon{on $\Gac \times (0,T)$}
   \comma i =1,2,3,
   \label{EQ09}
  \end{equation}
where $N$ denotes the time-independent outward normal vector with respect to~$\Omega_e$.
The boundary condition on $\Gab$ given in~\eqref{EQ03a} becomes
  \begin{equation}
   - a_{j i} q N_j
   + a_{j \ell} a_{k \ell} \partial_k v_i N_j
   = 0
   \inon{on $\Gab \times (0,T)$}
   \comma i =1,2,3,
   \label{EQ10}
  \end{equation}
where the outward normal is with respect to~$\Omf$. 
Finally, we also assume the no-slip boundary condition,
  \begin{equation}
   v = 0
   \inon{on $\Gaf \times (0,T)$}
   .
   \label{EQ11}
  \end{equation}
In this paper, we prove the following a~priori result.
\begin{theorem}
\label{T01}
Let $(v,w,q,\eta,a)$ be a smooth solution to our system on some time interval $[0,T)$, and set 
  \begin{equation}
   Y(t)
   =
   \|v\|_{H^3}^2
   + \|v_t\|_{H^2}^2
   + \|v_{tt}\|_{L^2}^2
   + \|w\|_{H^3}^2
   + \|w_t\|_{H^2}^2
   + \|w_{tt}\|_{H^1}^2
   + \|w_{ttt}\|_{L^2}^2
   .
   \label{EQ12}
  \end{equation}
Then there exists $C \geq 1$ and $\eps > 0$, independent of $T$, such that if $Y(0) \leq \eps$, then 
  \begin{equation}
   Y(t) 
   \leq
   C \eps \ee^{-t/C}
   ,
   \label{EQ13}
  \end{equation}
for $t \in [0,T)$.
\end{theorem}
We note that Theorem~\ref{T01} resolves two main challenges of the case of a general curved domain. The first challenge is to control the two quantities \eqref{two_suspects}. This is achieved by the vacuum bubble  $\Omb$ through the boundary condition \eqref{EQac}, which transfers the control of $q$ onto $\nabla v$. This also means that all energy level functionals must be involved in estimating the solution (see \eqref{def_Y_intro} below, where all operators $S$ are involved). This contrasts with \cite{KO}, which did not include $S=\mathrm{id}$ due to the need of the modified state variables, as mentioned above. In particular, the presence of all energy level functionals introduces a number of inter-dependencies between the energy levels, which makes us use weighted estimates between the energy levels (note the weights $\lambda^{-c_S}$ in \eqref{def_Y_intro} below). 

This also brings us to the second main challenge of this work, which arises specifically from considering a general curved domain. For such domains, the use of tangential operators $T_j$ (see \eqref{EQ15a}) results in a number of additional commutator terms (see \eqref{EQ95}--\eqref{EQ98}), where each of them needs to be handled using different energy levels. Consequently, this introduces further restrictions on the weights  of the energy levels (see \eqref{need3_intro}, \eqref{need1_intro}). In fact, different ways of estimating each term lead to different constraints on the weights. Thus, the problem of existence of ``the right estimates'' and of the  weights satisfying the resulting constraints, is the heart of the proof of Theorem~\ref{T01}. This is handled in detail in Section~\ref{sec_total}, see also Remark~\ref{Rem1} for an example of the difficulty.

We now discuss the sketch of the proof of Theorem~\ref{T01}, which exposes more details of the above challenges.

\section{Sketch of the proof of Theorem~\ref{T01}.}\label{sec_sketch}

For each $S\in \{ \id , T,\partial_t, T\partial_t, T^2, \partial_{tt}\} $ we consider two  energy estimates after applying the derivative $S$  to both the fluid and the elastic structure. One of the estimates is applied at the level of the velocity, i.e., multiplying the PDE for $Sv$ by $Sv$ and integrating, and then multiplying the PDE for $Sw$ by $Sw_t$ and integrating. The other one is concerned with the equipartition energy,  and is obtained by multiplying the PDE for $Sv$ by a special test function $\phi$, which is, roughly speaking, at the same  level as the trajectories $\eta$ of the velocity field $v$ (see~\eqref{EQ84}), and integrating, and by multiplying the PDE for $Sw$ by $w$ and integrating. We then  multiply the equipartition estimate by a (small) parameter $\lambda>0$ (to be determined below) and add to the estimate at the level of velocity to obtain a \emph{combined  energy estimate} of the form
\begin{equation}
   \frac{\d}{\d t} E_S(t,\tau) 
   + D_S(t)
   \les
   L_S(t,\tau) 
   + N_S(t,\tau) 
   + C_S(t,\tau)
   \label{EQ94_intro}
  \end{equation}
on  time interval $(\tau ,t)$ for each $S$, see  Section~\ref{sec_combined} for a careful derivation of~\eqref{EQ94_intro}. Here  $E_S$ stands for the \emph{pointwise (in time) energy}, $D_S$ stands for the \emph{dissipation energy}, $L_S$ stands for the \emph{linear terms}, $N_S$ stands for \emph{nonlinear terms}, and $C_S$ stands for the \emph{commutator terms}, see~\eqref{EQ95} for precise definitions. We then combine the energy estimates \eqref{EQ94_intro} as follows.\\

\noindent\texttt{Step 1.} Given any weights $c_S\geq 0$, where $S\in \{ \id , T,\partial_t, T\partial_t, T^2, \partial_{tt}\}$, we define the total weighted  energy of the system as
 \eqnb\label{def_Y_intro}
 Y(t)
    \coloneqq 
     \sum_{S }
      \lambda^{-c_S} \left(
       \|Sv\|_{L^2}^2
       + \|Sw_t\|_{L^2}^2
       + \|Sw\|_{H^1}^2
      \right),
\eqne
where we denoted the sum over all $S\in \{ \id , T,\partial_t, T\partial_t, T^2, \partial_{tt}\}$ by
\[
\sum_{S } \coloneqq \sum_{S \in \{ \id , T,\partial_t, T\partial_t, T^2, \partial_{tt}\}},
\]
a convention we apply throughout the paper, for brevity.\\

Here the operator $T$ denotes any one from a family of smooth vector fields $\{T_1,T_2,T_3\}$ which are defined globally on $\overline{\Omf \cup \Ome}$ and tangential to all boundaries $\Gaf$, $\Gab$, $\Gac$, and~$\Gae$, see~\cite[Theorem~3.1]{JKL} or~\cite{K}. 
They are of the form
  \begin{equation}
   T_j
   =
   \sum_{i=1}^3 b_{ij}(x) \partial_i
   \comma 
   1 \leq j \leq 3
   ,
   \label{EQ15a}
  \end{equation}
where the coefficients $b_{ij}$ are smooth time-independent functions.
Moreover, there exists a globally defined vector field $Y_0$ such that
  \begin{equation}
   \frac{\partial}{\partial x_k}
   = 
   \xi_k(x) Y_0
   + \sum_{j=1}^3 \eta_{jk}(x) T_j
   \comma
   1 \leq k \leq 3
   ,
   \label{EQ15b}
  \end{equation}
on a neighborhood of the boundaries for some smooth coefficients $\xi_k$ and~$\eta_{jk}$. 
By $T^2$ we mean any combination $T_{i_1}T_{i_2}$ where $i_1,i_2 \in \{1,2,3\}$.\\

Furthermore, we emphasize that the total energy \eqref{def_Y_intro} includes the case $S=\mathrm{id}$, and that it does not use any renormalization of the elastic displacement, contrary to the approach of~\cite{KO}. This is possible thanks to the Poincar\'e inequality \eqref{EQ41} for the pressure $q$, which is a consequence of the presence of the vacuum bubble $\Omb $ (recall \eqref{EQ03a}). In particular, this resolves the issue of the duality of the control of the quantities \eqref{two_suspects} described above.\\

\noindent\texttt{Step 2.} Given coefficients $c_S$, for $S\in \{ \id , T,\partial_t, T\partial_t, T^2, \partial_{tt}\}$, we define (in~\eqref{def_abc})  $\alpha, \kappa \in \RR$ , as well as, for technical reasons, a  number $\epsilon \in \RR$ (defined in~\eqref{epsilon}), so that, if 
\eqnb\label{need3_intro}
\epsilon >0 \andand c_{\p_t}>c_{T\p_t }-1,
\eqne
then the energy estimates \eqref{EQ94_intro} multiplied by the weights $\lambda^{-c_S}$, respectively, can be combined into the \emph{total energy estimate},
\eqnb\label{est_apr_intro}
   Y (t)
  + \lambda \int_\tau^t Y 
  \leq
  C(1 + \lambda^\alpha (t-\tau))Y(\tau)
  + C(\lambda^\alpha  + \lambda^{\kappa} (t-\tau )^2  ) \int_\tau^t Y
  + h O(Y)
\eqne
for all sufficiently small $\lambda >0$, and all $\tau \geq 0$, $t>\tau$ such that
\eqnb\label{def_h}
 h(t)
   := 
   \sup_{[0,t)} Y^{1/2}
   + \int_0^t Y^{1/2}
   \eqne
   is sufficiently small.\\

Here $O(Y)$ denotes any term of the form of sums of terms involving any constant, any powers of $\lambda$ and $(t-\tau)$ and at least two factors of the form $Y^{1/2}(\tau )$, $Y^{1/2} (t)$ and $\int_\tau^t Y^{1/2}$.\\

We emphasize that this step, which we verify in detail in Section~\ref{sec_total}, is the most technical part of the paper. In fact, the definitions of $\alpha,\kappa,\epsilon$  reflect the complexity of the system, i.e., the fact that the linear terms $L_S$, nonlinear terms $N_S$ and commutator terms  $C_S$ appearing in the estimates \eqref{EQ94_intro} for a given $S$ can only be handled  by employing the pointwise energy $E_{S'}$ and the dissipation energy $D_{S'}$ for some $S' \ne S$. The interactions appear at every level $S$, which makes the system technically challenging. \\

\noindent\texttt{Step 3.} We prove an ODE-type Lemma~\ref{L10}, which shows that \eqref{est_apr_intro} guarantees exponential decay of $Y(t)$ for sufficiently small $\lambda >0$ and sufficiently small $Y(0)$ if 
\eqnb\label{need1_intro}
\alpha >1, \,\, \kappa >3.
\eqne

The lemma is inspired by \cite[Lemma~3.1]{KO} and is proven in Section~\ref{sec_ode}. \\

\noindent\texttt{Step 4.} We conclude the proof.\\

Namely, it remains to show that there exists a choice of the coefficients $c_S\geq 0$, $S\in \{ \id , T,\partial_t, T\partial_t, T^2, \partial_{tt}\} $ such that \eqref{need3_intro} and \eqref{need1_intro} hold. This turns out to be true, by taking
\eqnb\label{the_choice_intro}
  \begin{aligned}
    c_{\id} &\coloneqq  5, & c_{\p_t} &\coloneqq  20/3,\\
    c_{T} &\coloneqq  5/3, & c_{T\p_t} &\coloneqq  17/3,\\
    c_{T^2} &\coloneqq  0, & c_{\p_{tt}}&\coloneqq  25/3,
  \end{aligned}
  \eqne
  see Section~\ref{sec_total} for a verification. Having fixed the coefficients $c_S$, we finally  fix  $\lambda >0$ to be sufficiently small, so that Step~3 concludes the proof of Theorem~\ref{T01}.

\section{Preliminaries}
\label{sec_prelim}
We will use the short-hand notation
\eqnb\label{shorthand}
\| \cdot \| \coloneqq \| \cdot \|_{L^2 },\quad \| \cdot \|_k \coloneqq \| \cdot \|_{H^k} \quad \text{ for }k\geq 1,
\eqne
where the norms on the right-hand sides are taken on either $\Ome$ or $\Omf$, depending on the context, i.e., depending whether the norm concerns $v$ or~$w$. We let
  \begin{align}
   \begin{split}
    X (t)
    &\coloneqq \sum_{S \in \{\id ,T,\partial_t, T\partial_t, T^2, \partial_{tt}\}}
    (
      \|Sv\|^2
      + \|Sw_t\|^2
      + \|Sw\|^2
    )
   \end{split}
   \label{EQ15}
  \end{align}
 denote the \emph{unweighted total energy} of the fluid-structure system and we denote by
\eqnb\label{def_g}
g(t) \coloneqq \sup_{[0,t)} X^{1/2} + \int_0^t X^{1/2}
\eqne 
 the \emph{total accumulation} of $X^{1/2}$ until time $t>0$. \\

In this section, we prove the following lemma.
\begin{lemma}
\label{L02} Suppose that $(v,w,q,\eta,a)$ is a smooth solution of the fluid-structure interaction system~\eqref{EQ01}--\eqref{EQ11}. For all sufficiently small $\gamma >0$ we have that 
 \eqnb
   \|v\|_{3}^2
   + \|q\|_{2}^2
   + \|v_t\|_{2}^2
   + \|v_{tt}\|^2
   + \|q_t\|_{1}^2
   + \|w\|_{3}^2
   + \|w_t\|_{2}^2 
   + \|w_{tt}\|_{1}^2
   + \|w_{ttt}\|^2
   \les 
   X
   \label{EQ14}
  \eqne
for all $t>0$ such that  $g(t) \leq \gamma$.
\end{lemma}
Before proving the lemma, we provide some important estimates on the Lagrangian map~$\eta$. 
First, note that $\eta_t = v$, so that 
  \begin{equation}
   \eta(x,t) 
   = 
   \eta(x,0)
   +\int_0^t \eta_t(x,s)\, \d s
   =
   x 
   + \int_0^t v(x,s) \, \d s
   .
   \label{EQ16}
  \end{equation}
By incompressibility, $a = (\nabla \eta)^{-1}$ is also its own cofactor matrix.
Therefore,
  \begin{equation}
   a_{i j} 
   = 
   \frac{1}{2} \eps_{i m n} \eps_{i k\ell} \partial_m \eta_k \partial_n \eta_\ell
   .
   \label{EQ19}
  \end{equation}
We thus have the following statements, the proofs of which are straightforward and can be found in~\cite[Section~2.1]{KO}.
\begin{proposition}
\label{P01}
We have 
  \begin{equation}
   \|I - \nabla \eta\|_{2}
   + \|D^2 \eta\|_{1}
   \les 
   \|v\|_{L^1((0,t); H^3)}
   \label{EQ17}
  \end{equation}
and
  \begin{align}
   \|a_t\|_{2}
   &\les 
   \|v\|_{L^\infty ((0,t); H^3)} (1 + \|v\|_{L^1((0,t); H^3)}),
   \label{EQ20}
   \\
   \|a_{tt}\|_{1}
   &\les
   \|\nabla \eta\|_{2} \|v_t\|_{2}
   + \|v\|_{3}^2
   .
   \label{EQ21}
  \end{align}
Also
  \begin{equation}
   \|I - a\|_{2}
   \les 
   \|v\|_{L^1((0,t); H^3)} (1 + \|v\|_{L^1((0,t); H^3)})
   \label{EQ25}
  \end{equation}
and 
  \begin{equation}
   \|I - aa^T\|_{2}
   \les
   \|I - a\|_{2} (1 + \|I - a\|_{2})
   .
   \label{EQ26}
  \end{equation}
\end{proposition}
In particular, Proposition~\ref{P01} implies that
\eqnb\label{i-a}
\| I-a \|_2, \| I - \nabla \eta \|_2 \les  g,
\eqne
  \begin{align}
   \begin{split}
    \|\partial_t ((I-a)q)\|_{1} 
    &\les 
    \|a_t\|_{2} \|q\|_{2}
    + \|I-a\|_{2} \|q_t\|_{1}
    \les 
    gX^{1/2}
    ,
    \\
    \|\partial_t ((I-aa^T)\nabla v)\|_{1}
    &\les
    \|a_t\|_{2}\|a\|_{2} \|v\|_{3}
    + \|I-aa^T\|_{2}\|v_t\|_{2}
    \les
    gX^{1/2}
    ,\\
    \|a_t\|_{H^2}, \|I-aa^T\|_{H^2} &\les g
   \end{split}
   \label{EQ102a}
  \end{align}
and 
  \begin{align}
   \begin{split}
    \|\partial_t((I-a)\nabla v)\|_{1}
    &\les
    \|a_t\|_{3} \|v\|_{3}
    + \|I-a\|_{2}\|v_t\|_{2}
    \les 
    X + g X^{1/2}
    ,
    \\
    \|\partial_{tt}((I-a)\nabla v)\|
    &\les
    \|I-a\|_{2}\|\nabla v_{tt}\|
    + \|a_t\|_{2}\|v_t\|_{1}
    + \|a_{tt}\|\,\|v\|_{3}
    \\&\les 
    g\|\nabla v_{tt}\|
    + X
    .
   \end{split}
   \label{EQ102b}
  \end{align}

Next, we state the standard $H^2$ and $H^3$ estimates for the Stokes system.

\begin{lemma}
Let $(u,p)$ be a solution of the problem
  \begin{align}
   \begin{split}
    -\Delta u + \nabla p 
    &= 
    f 
    \inon{in $\Omf$, }
    \\
    \nabla \cdot u 
    &= 
    k 
    \inon{in $\Omf$,}
    \\
    u 
    &= 
    h_1 
    \inon{on  $\Gaf$,}
    \\
    N_j \partial_j u + Np 
    &= 
    h_2  
    \inon{on $\Gac$,}
    \\
    N_j \partial_ju + Np 
    &= 
    h_3 
    \inon{on $\Gab$.}
   \end{split}
   \label{EQ30b}
  \end{align}
Then,
  \begin{equation}
   \|u\|_{2}
   + \|p\|_{1}
   \les
   \|f\| 
   + \|k\|_{1}
   + \|h_1\|_{H^{3/2}(\Gaf)}
   + \|h_2\|_{H^{1/2}(\Gac)}
   + \|h_3\|_{H^{1/2}(\Gab)}
   .
   \label{EQ30c}
  \end{equation}
If the boundary condition on $\Gac$ is replaced by $u = r$,  then we have 
  \begin{equation}
   \|u\|_{2}
   + \|\nabla p\|
   \les
   \|f\|
   + \|k\|_{1}
   + \|h_1\|_{H^{3/2}(\Gaf)}
   + \|r\|_{H^{3/2}(\Gac)}
   + \|h_3\|_{H^{1/2}(\Gab)}
   \label{EQ30d}
  \end{equation}
and 
  \begin{equation}
   \|u\|_{3}
   + \|\nabla p\|_{1}
   \les 
   \|f\|_{1}
   + \|k\|_{2}
   + \|h_1\|_{H^{5/2}(\Gaf)}
   + \|r\|_{H^{5/2}(\Gac)}
   + \|h_3\|_{H^{3/2}(\Gab)}
   .
   \label{EQ30e}
  \end{equation}
\end{lemma}
\noindent

For a proof of this lemma, see~\cite[Proposition~4.2]{LT} and also~\cite{G,GS,S}.

\begin{proof}[Proof of Lemma~\ref{L02}]
We first note that the assumption gives us 
  \begin{align}
   \begin{split}
    &\Big( 
      \|v\|_{L^1((0,t);H^3)}
      + \|v\|_{L^\infty((0,t);H^3)}
      + \|\nabla q\|_{L^\infty((0,t);H^1)}
    \\&\indeq\indeq\indeq\indeq
      +\|v_t\|_{L^1((0,t);H^2)}
      + \|q_t\|_{L^1((0,t);H^1)}
    \Big)
    (1+ \|v\|_{L^1((0,t);H^3)})
    \leq 
    \gamma
   \end{split}
   \label{EQ30z}
  \end{align}
if $\gamma > 0$ is sufficiently small. Indeed, \eqref{EQ30z} is valid for small times by continuity, and, provided $g(t)\leq \gamma$, the claim of the lemma and a continuity argument guarantees that \eqref{EQ30z} can be extended up to $(0,t)$. 
By Proposition~\ref{P01}, we also get
  \begin{equation}
   \|a_t\|_{2}
   + \|I-a\|_{L^\infty}
   + \|I - a a^T\|_{L^\infty}
   + \|I - a\|_{2}
   + \|I - aa^T\|_{2}
   + \|I - \nabla \eta\|_{2}
   + \|D^2 \eta\|_{1}
   \les 
   \gamma
   .
   \label{EQ29}
  \end{equation}
We rewrite our fluid system as 
  \begin{align}
   \begin{split}
    -\Delta v_i
    + \partial_i q
    &= 
    - \partial_t v_i
    - \partial_j ((\delta_{jk} - a_{j \ell} a_{k \ell}) \partial_k v_i)
    + \partial_k \left((\delta_{k i} - a_{k i})\left( q - \int_{\Omf} q \right) \right)
    ,
    \\
    \div v 
    &= 
    (\delta_{k i} - a_{k i}) \partial_k v_i
    .
   \end{split}
   \label{EQ31}
  \end{align}
For the boundary conditions, we have that $v = 0$ on $\Gaf$, along with 
  \begin{align}
   \begin{split}
    N_k \partial_k v_i 
    - N_i q
    &=
    N_j \partial_j w_i
    + (\delta_{j k} - a_{j \ell} a_{k \ell}) \partial_k v_i N_j
    - (\delta_{j i} - a_{j i})q N_j
    \inon{on $\Gac$}
    \comma i=1,2,3,
    \\
    N_k \partial_k v_i
    - N_i q
    &= 
    (\delta_{j k} - a_{j \ell} a_{k \ell}) \partial_k v_i N_j
    - (\delta_{j i} - a_{j i})q N_j
    \inon{on $\Gab$}
    \comma i=1,2,3,
   \end{split}
   \label{EQ32}
  \end{align}
where in the first line, the normal vector is with respect to $\Ome$, and in the second line, it is with respect to~$\Omf$. 

The $H^2$ estimate given by~\eqref{EQ30c} applied to~\eqref{EQ31}--\eqref{EQ32} gives
  \begin{align}
   \begin{split}
    \|v\|_{2}
    + \|\nabla q\|
    &\les 
    \|v_t\|
    + \|\nabla ((I-aa^T)\nabla v)\|
    + \left\| (I-a) \nabla \left(q - \int_{\Omf} q\right) \right\|
    \\&\indeq
    + \|v\|_{H^{3/2}(\Gac)}
    + \|(I-a) \nabla v\|_{H^{1/2}(\Gab)}
    + \|(I-a)q\|_{H^{1/2}(\Gab)}
    \\&\les
    \|v_t\|
    + \|v\|_{H^{3/2}(\Gac)}
    + \gamma (
      \|v\|_{H^2}
      + \|q\|_1
      )
    .
   \end{split}
   \label{EQ33}
  \end{align}
Furthermore, we can show that 
  \begin{equation}
   \|q\|
   \les 
   \|v\|_{2}
   + \|\nabla q \|
   .
   \label{EQ35a}
  \end{equation}
Indeed, note that 
  \begin{equation}
   \|q\|
   \les 
   \left\| q - \int_{\Gab}q \right\|
   + \int_{\Gab} |q|
   \les 
   \|\nabla q\|
   + \int_{\Gab} |q|
   ,
   \label{EQ36}
  \end{equation}
and that the boundary condition~\eqref{EQ10} may be rewritten as
  \begin{equation}
   qN_i
   =
   (\delta_{ji} - a_{j i})N_j q 
   + a_{j\ell} a_{k \ell} \partial_k v_i N_j
   \inon{on $\Gab$.}
   \label{EQ37}
  \end{equation}
Hence, 
  \begin{equation}
    \int_{\Gab} |q|
    = 
    \int_{\Gab} |qN|
    \les
    \int_{\Gamma_b} |I - a| |q|
    + \int_{\Gab} |aa^T| |\nabla v|
    \les 
    \|I-a\|_{2} \int_{\Gab} |q|
    + \|aa^T\|_{2}\|\nabla v\|_{1}
    ,
   \label{EQ38}
  \end{equation}
where we used the embedding $H^2(\Omf) \hookrightarrow C^{0,1/2}(\overline{\Omf})$ in the last line.
Next, note that $\|I - a\|_{2} < \gamma$ and 
  \begin{equation}
   \|aa^T\|_{2}
   \leq
   \|I - aa^T\|_{2} + \|I\|_{2}
   .
   \label{EQ39}
  \end{equation}
Thus we obtain
  \begin{equation}
   \int_{\Gab}|q|
   \les
   \gamma \int_{\Gab} |q|
   + (\gamma +1) \|v\|_{2}
   ,
   \label{EQ40}
  \end{equation}
and after absorbing the first term on the right-hand side, we get 
  \begin{equation}
   \|q\|
   \les 
   \|v\|_{2}
   + \|\nabla q\|
   ,
   \label{EQ41}
  \end{equation}
completing the proof of~\eqref{EQ35a}.

Using~\eqref{EQ35a} in~\eqref{EQ33} to replace $\|\nabla q\|$ with $\|q\|_1$, we absorb the term with $\gamma$ into the left-hand side to get
  \begin{align}
   \begin{split}
    \|v\|_{2}
    + \|q\|_1
    &\les 
    \|v_t\|
    + \|v\|_{H^{3/2}(\Gac)}
    \les
    \|v_t\|
    + \|v\|_{H^{1/2}(\Gac)}
    + \|T v\|_{H^{1/2}(\Gac)}
    \\&\les
    \|v_t\|
    + \|w_t\|_{1}
    + \|Tw_t\|_{1}
    \les
    \|v_t\|
    + \|\nabla w_t\|
    + \|\nabla Tw_t\|
    ,
   \end{split}
   \label{EQ34}
  \end{align}
where we used Poincar\'e's inequality in the last line, as $w = w_t  =0$ on~$\Gae$. 
We also used the fact that the tangential operators commute with the restriction. 
From here, we obtain
  \begin{equation}
   \|v\|_{2}
   + \|q\|_{1}
   \les 
   X
   .
   \label{EQ42}
  \end{equation}

Now, we list some elliptic estimates for $w$ in $H^2$ and~$H^3$. 
Recall that $w$ satisfies the equation
  \begin{equation}
   -\Delta w 
   = 
   -\alpha w_t
   - w_{tt}
   \inon{in $\Ome \times (0,T)$.}
   \label{EQ43eq}
  \end{equation}
Considering the $H^2$ Dirichlet estimate, we have 
  \begin{equation}
   \|w\|_{2}
   \les 
   \|w_t\|
   + \|w_{tt}\|
   + \|w\|_{H^{3/2}(\Gamma_c)}
   \les 
   \|w_t\|
   + \|w_{tt}\|
   + \|\nabla w\|
   + \|\nabla T w\|
   ,
   \label{EQ43}
  \end{equation}
where we used Poincar\'e's inequality.
Additionally, we get the following  $H^3$ estimate:
  \begin{equation}
   \|w\|_{3}
   \les 
   \|w_t\|_{1}
   + \|w_{tt}\|_{1}
   + \|w\|_{H^{5/2}(\Gamma_c)}
   \les 
   \|w_t\|_{1}
   + \|w_{tt}\|_{1}
   + \|\nabla w\|
   + \|\nabla T^2 w\|
   .
   \label{EQ44}
  \end{equation}
Also, taking a time-derivative of the elliptic system~\eqref{EQ43} gives $-\Delta w_t = -\alpha w_{tt} - w_{ttt}$.
The $H^2$ Dirichlet estimate then gives us 
  \begin{equation}
   \|w_t\|_{2}
   \les
   \|w_{ttt}\|
   + \|w_{tt}\|
   + \|w_t\|_{H^{3/2}(\Gamma_c)}
   \les
   \|w_{ttt}\|
   + \|w_{tt}\|
   + \|v\|_{2}
   .
   \label{EQ45}
  \end{equation}
From the inequalities~\eqref{EQ43}--\eqref{EQ45}, we surmise that 
  \begin{equation}
   \|w\|_{3}^2
   \les
   X
   .
   \label{EQ46}
  \end{equation}

Next, we derive an $H^3$ estimate for~$v$. 
From the system~\eqref{EQ31}--\eqref{EQ32} and the Dirichlet estimate~\eqref{EQ30e}, we acquire
  \begin{align}
   \begin{split}
    \|v\|_{3}
    + \|q\|_{2}
    &\les 
    \|v_t\|_{1}
    + \|\nabla ((I - aa^T) \nabla v)\|_{1}
    + \left\| (I-a) \nabla \left( q - \int_{\Omf} q \right) \right\|_{1}
    \\&\indeq
    + \|(I-a) \nabla v\|_{2}
    + \|v\|_{H^{5/2}(\Gac)}
    + \|(I-a) \nabla v\|_{H^{3/2}(\Gab)}
    + \|(I - a)q\|_{H^{3/2}(\Gab)}
    \\&\les
    \|v_t\|_{1}
    + \|v\|_{H^{5/2}(\Gac)}
    + \gamma (
      \|v\|_{3}
      + \|q\|_{2}
      )
    ,
   \end{split}
   \label{EQ47}
  \end{align}
where we used~\eqref{EQ35a} to get $\|q\|_2$ on the left-hand side.
Next, we absorb the $\gamma$ term into the left-hand side to get
  \begin{equation}
   \|v\|_{3}
   + \|q\|_{2}
   \les 
   \|v_t\|_{1}
   + \|v\|_{H^{5/2}(\Gac)}
   \les
   \|v_t\|_{1}
   + \|v\|_{1}
   + \|\nabla T^2v\|
   .
   \label{EQ48}
  \end{equation}

Next, we control $Tv$ in~$H^2$.
Applying a first order tangential operator $T$ to~\eqref{EQ31}--\eqref{EQ32} and tracking the commutator terms, we obtain the system
  \begin{align}
   \begin{split}
    -\Delta T v_i
    + \partial_i Tq
    &= 
    - T\partial_t v_i 
    - T\partial_j((\delta_{jk} - a_{jk}) \partial_k v_i)
    + T\partial_k \left( (\delta_{ki} - a_{ki} )\left(q - \int_{\Omf} q\right)\right)
    \\&\indeq
    + (T\Delta v_i - \Delta Tv_i)
    - (T\partial_i q - \partial_i Tq)
    ,
    \\
    \div Tv 
    &=
    T((\delta_{ki} - a_{ki})\partial_k v_i) 
    - (T \div v - \div Tv)
    ,
   \end{split}
   \label{EQ50}
  \end{align}
with the boundary conditions
  \begin{align}
   \begin{split}
    N_k \partial_k Tv_i 
    - N_i Tq
    &=
    T(N_j \partial_j w_i)
    + T((\delta_{k i} - a_{j \ell} a_{k\ell} ) \partial_k v_i N_j)
    - T((\delta_{j i} - a_{ji})q N_j)
    \\&\indeq
    - (T(N_k \partial_k v_i) - N_k \partial_k Tv_i)
    + (T(N_i q) - N_i Tq)
    \inon{on $\Gac$,}
    \\
    N_k \partial_k Tv_i - N_i T q
    &= 
    T((\delta_{kj} - a_{kj})\partial_k v_i N_j)
    - (T(N_k \partial_k v_i) - N_k \partial_k Tv_i)
    + (T(N_i q) - N_i Tq)
    \inon{on $\Gab$.}
   \end{split}
   \label{EQ51}
  \end{align}
From the Neumann estimate~\eqref{EQ30d}, we then have
  \begin{align}
   \begin{split}
    \|Tv\|_{2}
    + \|\nabla Tq\|
    &\les
    \|Tv_t\|
    + \|T\nabla((I-aa^T) \nabla v)\|
    + \left\|T \left( ( I-a) \nabla \left(q - \int_{\Omf} q\right)\right)\right\|
    \\&\indeq\indeq
    + \|T((I-a)\nabla v)\|_{1}
    + \|w\|_{3}
    + K_{\text f}
    + K_{\text b}
    + K_{\text c}
    ,
   \end{split}
   \label{EQ52}
  \end{align}
where 
  \begin{align}
   \begin{split}
    K_{\text f}
    &=
    \|T \Delta v - \Delta Tv\|
    + \|T \nabla q - \nabla T q\|
    + \|T \div{v} - \div{Tv}\|_{1},
    \\
    K_{\text b}
    &=
    \|T((I-a)\nabla v N)\|_{H^{1/2}(\Gab)}
    + \|T(N\cdot \nabla v) - N \cdot \nabla Tv\|_{H^{1/2}(\Gab)}
    + \|T (Nq) - NTq\|_{H^{1/2}(\Gab)},
    \\
    K_{\text c}
    &= 
    \|T((I - aa^T)\nabla v N)\|_{H^{1/2}(\Gac)}
    + \|T((I - a)qN)\|_{H^{1/2}(\Gac)}
    + \|T(N \cdot \nabla v) - N \cdot \nabla Tv\|_{H^{1/2}(\Gac)}
    \\&\indeq
    + \|T(Nq) - N Tq\|_{H^{1/2}(\Gac)}.
   \end{split}
   \label{EQ53}
  \end{align}
Consequently,
  \begin{equation}
   K_{\text f}
   + K_{\text b}
   + K_{\text c}
   \les
   \|v\|_{2}
   + \|q\|_{1}
   + \gamma (\|v\|_{3} + \|q\|_{2}),
   \label{EQ54}
  \end{equation}
where we used $\|v\|_{2} + \|q\|_{1} \les \gamma$.
Summarizing, we obtain
  \begin{equation}
   \|Tv\|_{2}
   + \|\nabla Tq\|
   \les 
   \|Tv_t\|
   + \|w\|_{3}
   + \gamma (
   \|v\|_{3}
   + \|q\|_{2}
    )
    + \|v\|_{2}
    + \|q\|_{1}
    .
   \label{EQ59}
  \end{equation}
Using the bound~\eqref{EQ59} in~\eqref{EQ48} and absorbing the $\gamma$ term, we arrive at 
  \begin{equation}
   \|v\|_{3}
   + \|q\|_{2}
   \les
   \|v_t\|_{1}
   + \|Tv_t\|
   + \|w\|_{3}
   + \|v\|_{2} + \|q\|_{1}
   .
   \label{EQ60}
  \end{equation}
Besides $\|v_t\|_{1}$, every term on the right-hand side is controlled by~$X^{1/2}$.

Next, we perform the $H^2$ estimates on~$v_t$.
Applying $\partial_t$ to~\eqref{EQ31}--\eqref{EQ32} and noting that $[\partial_t, T] = 0$ since the coordinate functions of the vector fields $T_i$ are time-independent, we obtain
  \begin{align}
   \begin{split}
    -\Delta \partial_t v_i 
    + \partial_iq_t
    &=
    -\partial_{tt} v_i
    -\partial_t \partial_j((\delta_{jk} - a_{j\ell}a_{k\ell}) \partial_k v_i)
    + \partial_t ((\delta_{ki} -a_{ki})\partial_k q)
    ,
    \\
    \div v_t 
    &= 
    \partial_t \partial_k (\delta_{ki} - a_{ki} v_i)
    ,
   \end{split}
   \label{EQ61}
  \end{align}
with the boundary conditions
  \begin{align}
   \begin{split}
    N_k \partial_k \partial_t v_i
    - N_i q_t
    &=
   \partial_t((\delta_{jk} - a_{j\ell}a_{k\ell}) \partial_k v_iN_j)
   -\partial_t ((\delta_{ji} - a_{j i} )qN_j)
   + \partial_t(N_j \partial_j w_j)
   \inon{on $\Gac$,}
   \\
   N_k \partial_k v_i
   - N_i q_t
   &= 
   \partial_t((\delta_{kj} - a_{kj}) \partial_k v_i N_j)
   - \partial_t ((\delta_{ji} - a_{ji})qN_j)
   \inon{on $\Gab$.}
   \end{split}
   \label{EQ62}
  \end{align}
The Stokes estimate~\eqref{EQ30c} with the Neumann condition on $\Gac $ yields
  \begin{align}
   \begin{split}
    \|v_t\|_{2}
    + \|\nabla q_t\|_{1}
    &\les
    \|v_{tt}\|
    + \|\partial_t \nabla ((I-aa^T)\nabla v)\|
    + \|\partial_t ((I-a)\nabla q)\|
    \\&\indeq
    +\|\partial_t ((I-a)\nabla v)\|_{1}
    + \|\nabla w_t\|_{H^{1/2}(\Gac)}
    + \|\partial_t((I-aa^T)\nabla v)\|_{H^{1/2}(\Gac)}
    \\&\indeq
    + \|\partial_t((I-a)q)\|_{H^{1/2}(\Gac)}
    + \|\partial_t((I-a)\nabla v)\|_{H^{1/2}(\Gab)}
    + \|\partial_t((I-a)q)\|_{H^{1/2}(\Gab)}
    \\&\les
    \|v_{tt}\| 
    + \|\partial_t(aa^T)\|_{2} \|v\|_{2}
    + \|I-a a^T\|_{2} \|v_t\|_{2}
    + \|a_t\|_{2}\|v\|_{2}
    \\&\indeq
    + \|I-a\|_{2} \|v_t\|_{2}
    + \|a_t\|_{2} \|q\|_{1}
    + \|I-a\|_{2} \|q_t\|_{1}
    + \|w_t\|_{2}
    .
   \end{split}
   \label{EQ64}
  \end{align}
In the last line we used the trace inequality for the boundary terms, and, except for $w_t$, their bounds were identical to that of the interior terms. 
Recall that 
  \begin{equation}
   \|\partial_t (aa^T)\|_{2}
   \les 
   \|a_t\|_{2} \|a\|_{2}
   \les 
   \|\nabla \eta\|_{2} \|v\|_{3} (1 + \|I-a\|_{2})
   \les 
   \|v\|_{3}
   ,
   \label{EQ65}
  \end{equation}
along with $\|a_t\|_{2} \les \|v\|_{3}$.
Repeating the proof of~\eqref{EQ35a}, we use~\eqref{EQ65} to obtain
  \begin{equation}
   \|q_t\|_{1}
   \les
   \|v_t\|_{2}
   + \|\nabla q_t\|
   + \|v\|_{H^3} \|v\|_{H^2}
   ,
   \label{EQ63}
  \end{equation}
and combining~\eqref{EQ65} and~\eqref{EQ63} with~\eqref{EQ64}, we have 
  \begin{equation}
  \|v_t\|_{2}
  + \|q_t\|
  \les 
  \|v_{tt}\|
  + \gamma (\|v_t\|_{2} + \|q_t\|_{1})
  + \|v\|_{3} (\|v\|_{2} + \|q\|_{1})
  + \|w_t\|_{2}
  .
  \label{EQ66}
  \end{equation}
From~\eqref{EQ45}, we know that
  \begin{equation}
   \|w_t\|_{2}
   \les
   \|w_{ttt}\|
   + \|w_{tt}\|
   + \|v\|_{2}
   ,
   \label{EQ66a}
  \end{equation}  
giving us
  \begin{equation}
   \|v_t\|_{2}
   + \|q_t\|_{1}
   \les
   \|v_{tt}\|
   + \|v\|_{3} (\|v\|_{2} + \|q\|_{1})
   + \|w_{ttt}\|
   + \|w_{tt}\|
   + \|v\|_{2}
   .
   \label{EQ67}
  \end{equation}
We also note  the following $H^2$ Dirichlet estimate:
  \begin{equation}
   \|v_t\|_{2}
   + \|q_t\|_{1}
   \les
   \|v_{tt}\|
   + \|v_t\|_{1}
   + \|\nabla T v_t\|
   + \|v\|_{3} (\|v\|_{2} + \|q\|_{1})
   .
   \label{EQ67a}
  \end{equation}
Using~\eqref{EQ67},~\eqref{EQ44}, and~\eqref{EQ34} on our most recent rendition of the $H^3$ estimate on $v$ in~\eqref{EQ60}, we arrive at
  \begin{align}
   \begin{split}
    \|v\|_{3} 
    + \|q\|_{2}
    &\les
    \|v_t\|
    + \|v_{tt}\|
    + \|v\|_{3}(\|v\|_{2} + \|q\|_{1})
    + \|Tv_t\|
    \\&\indeq
    + \|w\|_{1}
    + \|w_t\|_{1}
    + \|w_{tt}\|_{1}
    + \|\nabla Tw_t\|
    + \|\nabla T^2 w\|
    .
    \end{split}
    \label{EQ68}
   \end{align}
Finally, recall that $\|v\|_{2} + \|q\|_{1} \les \gamma$ due to the smallness condition~\eqref{EQ30z}.
Thus we may absorb the quadratic term to obtain 
  \begin{equation}
   \|v\|_{3}^2
   + \|q\|_{2}^2
   \les 
   X
   .
   \label{EQ69}
  \end{equation}
Hence,
  \begin{equation}
   \|v\|_{3}^2
   + \|q\|_{2}^2
   + \|v_t\|_{2}^2
   + \|q_t\|_{1}^2
   + \|w\|_{3}^2
   + \|w_t\|_{2}^2
   + \|w_{tt}\|_{1}^2
   + \|w_{ttt}\|^2
   \les 
   X
   .
   \label{EQ70}
  \end{equation}
The proof is complete after setting $\gamma$ small enough for all of our absorption arguments to hold. 
\end{proof}

We conclude the section by a duality estimate
for the integral $\int_{\Gamma}Tf g\,d\sigma$,
where $\Gamma=\Gac \cup \Gaf \cup \Gab$ or $\Gamma = \Gac\cup \Gae$,
needed in the sequel.

\begin{Lemma}
\label{L01}
For $f,g\in C^{\infty}(\overline\Omega)$, we have
  \begin{equation}
   \left|\int_{\Gamma} Tf g\,d\sigma
   \right|
   \leq
   \Vert f\Vert_{H^{1/2}(\Gamma)}
   \Vert g\Vert_{H^{1/2}(\Gamma)}
   .
   \label{EQ18}
  \end{equation}
\end{Lemma}

\begin{proof}[Proof of Lemma~\ref{L01}]
The proof is obtained by using a partition of unity and straightening
of the boundary. 
Consider the mapping as in Figure~\ref{F2},
  \begin{figure}[ht]
    \centering
    \def\svgwidth{1.9\columnwidth}
    \incfig[0.85\columnwidth]{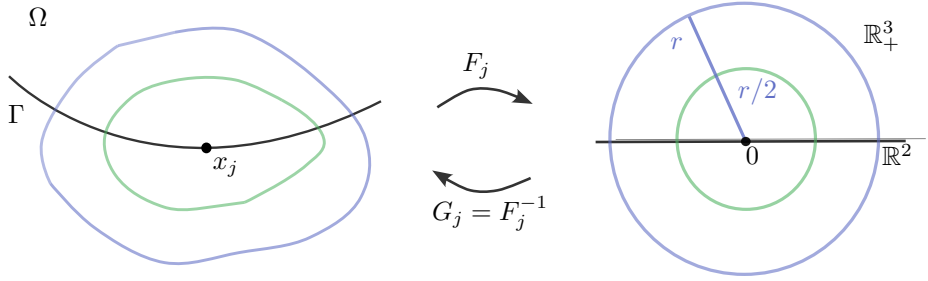}
    \caption{A local mapping}
    \label{F2}
  \end{figure}
for $j=1,\ldots,m$, where $m\in\mathbb{N}$.

For any fixed $j$, assume that~$\psi_j,\zeta_j\in C_{0}^\infty(\mathbb{R}^{3},[0,1])$ are supported in
$B_r$
such that
$\psi_j\equiv 1$ on~$B_{r/2}$ and $\zeta_j\equiv1 $ on a neighborhood
of~$\supp \psi_j$.
Define $\tilde{\psi}_j = \psi_j \circ F_j$ and $\tilde{\zeta}_j = \zeta_j \circ F_j$.
We assume that 
$ \{\tilde{\psi}_j\}_j $ forms a partition of unity for a neighborhood
of~$\Gamma$.

Now, using the partition-of-unity property, we have
  \begin{equation}
   \int_{\Gamma} Tf g 
   = \int_\Gamma T\left(
           \sum_{j} \tilde{\psi}_j  f
           \right)
	   g
   =
   \sum_{j} \int_{\Gamma} \tilde{\psi}_j  T f g
   +
   \sum_{j} \int_{\Gamma} T  \tilde{\psi}_j  f g
   .
   \llabel{EQ22}
  \end{equation}
Clearly, the second term satisfies the desired bound.
For a fixed $j$, denote
  \begin{equation}
   I
   = \int_{\Gamma}  \tilde{\psi}_j  T f g
   = \int_{\Gamma}  \tilde{\psi}_j  T( \tilde{\zeta}_j f)  (\tilde{\zeta}_j  g)
   .
   \llabel{EQ23}
  \end{equation}
Recall that any one $T$ from the collection defined in~\eqref{EQ15a} is of the form $T = b \cdot \nabla_x$. 
Changing coordinates, we have 
  \begin{equation}
   I
   =
   \int_{\mathbb{R}^2}
    \psi_j
    \bar{T}(\zeta_j \bar{f})
    \zeta_j \bar{g}
    J
   \d y
   ,
   \label{EQ24}
  \end{equation}
where
  \begin{equation}
    J = |\partial_{y_1} G_j(y_1,y_2,0) \times \partial_{y_2} G_j (y_1,y_2,0)|
    .
    \llabel{EQ24aa}
  \end{equation}
In the above, we defined $\bar{T} = ((DG_j)^{-1} b\circ G_j) \cdot \nabla_y$ and $\bar{f} = f \circ G_j$. 
Since $ b$ is tangential to $\Gamma$, it follows that $((DG_j)^{-1} b\circ G_j) \cdot e_3 = 0$.
Hence $\bar{T}$ must be of the form $\bar{T} = a_1(y) \partial_1 + a_2(y) \partial_2$, where $a_1$ and $a_2$ are smooth. 
Without loss of generality, assume that $\bar{T} = a_1(y) \partial_1$.
Hence
  \begin{align}
   \begin{split}
    |I|
    &=
    \Bigg| 
    \int_{\mathbb{R}^2} \psi_j 
    a_1
    \partial_1( \zeta_j \bar{f}) 
    \zeta_j \bar{g} 
    J 
    \d y
    \Bigg|
    \\&=
    \Bigg|
    \int_{\mathbb{R}^2} 
    \Lambda^{-1/2} \partial_1(\zeta_j \bar{f})
    \Lambda^{1/2} (\psi_j a_1 \zeta_j \bar{g} J)
    \d y
    \Bigg|
    \\&\les
    \|\zeta_j \bar{f}\|_{H^{1/2}(\mathbb{R}^2)}
    \|\psi_j a_1 \zeta_j \bar{g} J\|_{H^{1/2}(\mathbb{R}^2)}
    \les
    \|f\|_{H^{1/2}(\Gamma)} \|g\|_{H^{1/2}(\Gamma)}
    ,
   \end{split}
   \label{EQ28}
  \end{align}
where $\Lambda = (-\Delta)^{1/2}$ on~$\mathbb{R}^2$.
This completes the proof.
\end{proof}

\section{Combined energy estimates}\label{sec_combined}
In this section we prove the estimate \eqref{EQ94_intro}, i.e.,
\begin{equation}
   \frac{\d}{\d t} E_S(t,\tau) 
   + D_S(t)
   \les
   L_S(t,\tau) 
   + N_S(t,\tau) 
   + C_S(t,\tau)
   \label{EQ94}
  \end{equation}
for all sufficiently small $\lambda >0$, where 
 \begin{align}
   \begin{split}
    E_S(t,\tau)
    &\coloneqq 
    \frac{1}{2} \|Sv\|^2
    + \frac{1}{2} \|Sw_t\|^2
    + \frac{1}{2}\|\nabla Sw\|^2
    + \frac{\lambda \alpha}{2} \|Sw\|^2
    \\&\indeq
    + \frac{\lambda}{2} \|\nabla S(\eta - \eta(\tau))\|^2
    + \lambda \int_{\Omf} Sv \cdot \phi 
    + \lambda \int_{\Ome} Sw_t \cdot Sw 
    ,
    \\
    D_S(t)
    &\coloneqq 
    \frac{1}{2} \|\nabla S v\|^2
    + \frac{1}{C} \|Sv\|^2
    + (\alpha - \lambda) \|Sw_t\|^2
    + \frac{\lambda}{2}\|\nabla Sw\|^2
    + \frac{\lambda}{C} \|Sw\|^2
    ,
    \\
    L_S(t,\tau)
    &\coloneqq  
    -\lambda \int_{\Omf} \nabla Sv \colon \nabla \widetilde{S(w(\tau))}
    + \lambda \int_{\Omf} Sq \div{\phi}
    \\&=: 
    L_{S,1}(t,\tau)
    + L_{S,2}(t,\tau)
    ,
    \\
    N_S(t,\tau)
    &\coloneqq  
    \int_{\Omf} (\delta_{kj} - a_{j\ell}a_{k\ell}) \partial_k Sv_i \partial_j Sv_i
    + \lambda\int_{\Omf} S((\delta_{jk} - a_{j\ell}a_{k\ell})\partial_k v_i)\partial_j \phi_i
    \\&\indeq
    -\lambda \int_{\Omf} S((\delta_{ki} - a_{ki})q)\partial_k \phi_i
    ,\\
    C_S(t,\tau)
    &\coloneqq  C_{S,1}(t,\tau)
   + C_{S,2}(t),
   \end{split}
   \label{EQ95}
  \end{align}
with
  \begin{align}
   \begin{split}
    C_{S,1}(t,\tau)
    &\coloneqq 
    -\lambda \int_{\Omf} (S\partial_j v_i - \partial_j Sv_i) \partial_j \phi_i
    + \lambda \int_{\Omf} (S\partial_j(a_{j\ell} a_{k \ell} \partial_k v_i) - \partial_j S(a_{j\ell}a_{k\ell} \partial_k v_i))\phi_i
    \\&\indeq
    - \lambda \int_{\Omf} (S\partial_k (a_{k i} q) - \partial_k S(a_{k i} q))\phi_i
    + \lambda \int_{\Ome}(S\Delta w -\Delta Sw)\cdot Sw
    - \lambda \int_{\Gac} (S(N_j \partial_j w_i) - N_j\partial_j Sw_i)\phi_i
    \\&\indeq
    + \lambda \int_{\Gab\cup \Gac} (S(N_j^f a_{j \ell} a_{k\ell} \partial_k v_i) - N_j^f S(a_{j\ell} a_{k\ell} \partial_k v_i))\phi_i
    - \lambda \int_{\Gab \cup \Gac} (S(N_k^f a_{ki} q) - N_k^f S(a_{ki} q))\phi_i
    ,\\
    C_{S,2}(t)
    &\coloneqq 
    \int_{\Omf} (S\partial_j (a_{j\ell} a_{k\ell} \partial_k v_i) - \partial_j (a_{j\ell} a_{k\ell} \partial_k Sv_i)) Sv_i
    - \int_{\Omf} (S(a_{ki} \partial_k q) - a_{ki} \partial_k Sq) Sv_i
    \\&\indeq
    + \int_{\Gab \cup \Gac} (S(N_j a_{j\ell} a_{k\ell} \partial_k v_i) - N_j a_{j\ell} a_{k\ell} \partial_k Sv_i)Sv_i
    + \int_{\Gab\cup \Gac} (S(a_{ji} N_j q) -a_{ji}N_j Sq)Sv_i
    \\&\indeq
    + \int_{\Omf} Sq a_{ki} \partial_k Sv_i
    + \int_{\Ome} (S \Delta w - \Delta S w) \cdot Sw_t
    +\int_{\Gac} (S(N_j \partial_j w_i) - N_j \partial_j Sw_i) Sv_i
    ;
   \end{split}
   \label{EQ98}
  \end{align}
here, for each $S$, we have defined the test function
  \begin{equation}
   \phi(t) 
   \coloneqq 
   S \eta(t)
   - S(\eta(\tau))
   + \widetilde{S(w(\tau))}
   \inon{on $\Omf$},
   \label{EQ84}
  \end{equation}
in other words,
\[
\phi (t) = \begin{cases} S \eta(t)
   - S\eta(\tau)
   + \widetilde{Sw}(\tau)\hspace{2cm}&S\in \{ \id, T, T^2 \},\\ 
   v(t) & S= \p_t,\\
   Tv(t) & S= T\p_t,\\
   \p_t v(t) & S= \p_{tt}.   
   \end{cases}
\]
Here,  $\tilde{~}\colon H^k(\Ome) \to H^k(\Omf \cup \Ome)$ is an extension operator such that  $\|\widetilde{f}\|_{H^k} \les \|f\|_{H^k(\Ome)}$ for $k=0,1,2,3$. \\

Note that integrating the continuity condition~\eqref{EQ08} in time gives us 
  \begin{equation}
   \eta(t)
   -\eta(\tau)
   +w(\tau)
   =
   w(t)
   \inon{on $\Gac$,}
   \label{EQ85}
  \end{equation}
so that 
  \begin{equation}
   \phi(t)|_{\Gac} 
   = 
   Sw(t)|_{\Gac} 
   = 
   \widetilde{Sw}(t)|_{\Gac}
   .
   \label{EQ86}
  \end{equation}
Moreover, applying Young's inequality on the last two terms in the definition of $E_S$, we get the equivalence
  \begin{equation}
   E_S(t,\tau)
   \sim 
   \|Sv(t)\|^2
   + \|Sw_t\|^2
   + \|Sw(t)\|_{1}^2
   + \lambda \|\nabla S(\eta - \eta(\tau))\|^2
   + \lambda \int_{\Omf} Sv \cdot \widetilde{S(w(\tau))}
   ,
   \label{EQ95a}
  \end{equation}
which holds for all sufficiently small $\lambda>0$.

\begin{proof}[Proof of \eqref{EQ94}.]
We start by applying $S$ to the momentum equation~\eqref{EQ06}$_1$ to obtain
  \begin{align}
   \begin{split}
    & 
    S\partial_t v_i
    - \partial_j (a_{j\ell} a_{k\ell} \partial_k Sv_i)
    + a_{ki} \partial_k Sq
    \\&\indeq\indeq=
    (S \partial_j (a_{j\ell} a_{k\ell} \partial_k v_i)
    - \partial_j (a_{j \ell} a_{k \ell} \partial_k Sv_i))
    - (S(a_{ki} \partial_k q) - a_{ki} \partial_k Sq)
    .
   \end{split}
   \label{EQ72}
  \end{align}
For brevity, we denote the right-hand side of~\eqref{EQ72} as $(K_1^S)_i$. 
Testing with $Sv_i$ and integrating by parts, we get 
  \begin{align}
   \begin{split}
    \frac{1}{2}\frac{\d}{\d t}\|Sv\|^2
    + \|\nabla Sv\|^2
    &=
    \int_{\Omf} (\delta_{jk} - a_{j\ell}a_{k\ell}) \partial_k Sv_i \partial_j Sv_i
    + \int_{\Omf} Sq \partial_k (a_{ki} Sv_i)
    + \int_{\Omf} K_1^S \cdot Sv
    \\&\indeq
    + \int_{\partial \Omf} N_j a_{j \ell} a_{k\ell} \partial_k Sv_i Sv_i
    - \int_{\partial\Omf} N_k Sq a_{ki} Sv_i
    .
   \end{split}
   \label{EQ73}
  \end{align}
It will be advantageous in later estimates to use the Neumann boundary conditions~\eqref{EQ09}--\eqref{EQ10} in the boundary integrals above. 
That is, note that~\eqref{EQ09} gives us 
  \begin{equation}
    N_j a_{j\ell} a_{k\ell} \partial_k Sv_i
    - a_{ji} N_j Sq
    = 
    N_j \partial_j Sw_i
    + K_2^S
    \inon{on $\Gac$}
    \comma i=1,2,3,
    \label{EQ74}
  \end{equation}
where $K_2^S$ is the commutator 
  \begin{align}
   \begin{split}
    (K_2^S)_i
    &=
    (S(N_j \partial_j w_i) - N_j \partial_j Sw_i)
    + (S(a_{ji} N_j q) - a_{ji} N_j Sq)
    \\&\indeq
    + (S(N_j a_{j\ell} a_{k\ell} \partial_k v_i) - N_j a_{j\ell} a_{k\ell} \partial_k Sv_i)
    .
   \end{split}
   \label{EQ75}
  \end{align}
Also, on $\Gab$, we have
  \begin{equation}
   N_j a_{j\ell} a_{k \ell} \partial_k Sv_i
   - N_j a_{ji} S q
   =
   (K_3^S)_i
   \inon{on $\Gab$}
   \comma i=1,2,3,
   \label{EQ76}
  \end{equation}
where 
  \begin{equation}
   (K_3^S)_i
   =
   (S(N_j a_{j \ell} a_{k\ell} \partial_k v_i) 
   - N_j a_{j\ell} a_{k\ell} \partial_k Sv_i)
   + (S(N_j a_{ji} q) - N_j a_{ji} Sq)
   .
   \label{EQ77}
  \end{equation}
Substituting~\eqref{EQ74} and~\eqref{EQ76} into~\eqref{EQ73}, we obtain  
  \begin{align}
   \begin{split}
    &\frac{1}{2} \frac{\d}{\d t} \|Sv\|^2
    + \|\nabla Sv\|^2
    \\&\indeq\indeq=
    \int_{\Omf} Sq \partial_k (a_{ki} Sv_i)
    + \int_{\Omf} (\delta_{jk} - a_{j\ell} a_{k\ell})\partial_k Sv_i \partial_j Sv_i
    - \int_{\Gac} N_j^e \partial_j Sw_i S(w_t)_i
    \\&\indeq\indeq\indeq
    + \int_{\Omf} K_1^S \cdot Sv  
    + \int_{\Gac} K_2^S \cdot Sv 
    + \int_{\Gab} K_3^S\cdot  Sv
    ,
   \end{split}
   \label{EQ78}
  \end{align}
where we used the continuity of velocity condition~\eqref{EQ08} and define $-\Ne = \Nf$ where $\Ne$ denotes the outward normal vector with respect to $\Omega_e$ and analogously with~$\Nf$.
Except when indicated for emphasis, we take $N$ to mean the normal vector relative to $\Ome$ when we are on $\Gae$ and relative to $\Omf$ when on~$\Gab$.

Now, we derive estimates for the elastic displacement~$w$.
We apply $S$ to~\eqref{EQ01} and obtain 
  \begin{equation}
   S w_{tt}
   - \Delta S w
   + \alpha Sw_t
   =
   (S \Delta w - \Delta Sw)
   .
   \label{EQ79}
  \end{equation}
Testing with $Sw_t$ and integrating by parts, we find that
  \begin{align}
   \begin{split}
    &\frac{1}{2}\frac{\d}{\d t} \|Sw_t\|^2
    + \frac{1}{2} \frac{\d}{\d t} \|\nabla Sw\|^2
    + \alpha \|Sw_t\|^2
    \\&\indeq\indeq=
    \int_{\Gac} \Ne_i \partial_j Sw_i Sv_i
    + \int_{\Ome} (S\Delta w - \Delta Sw) \cdot Sw_t
    ,
   \end{split}
   \label{EQ80}
  \end{align}
where we used the continuity of velocity condition on~$\Gac$.
Now, we add equations~\eqref{EQ78} and~\eqref{EQ80} together to get
  \begin{align}
   \begin{split}
    &\frac{1}{2}\frac{\d}{\d t} (
        \|Sv\|^2
        + \|Sw_t\|^2
        + \|\nabla Sw \|^2
        )
    + \|\nabla Sv\|^2
    + \alpha \|Sw_t\|^2
    \\&\indeq\indeq= 
    \int_{\Omf} (\delta_{jk} - a_{j\ell}a_{k\ell}) \partial_k Sv_i \partial_j Sv_i
    + \int_{\Omf} Sq \partial_k (a_{ki} Sv_i) 
    + \mathcal{K}_1^S
    ,
   \end{split}
   \label{EQ81}
  \end{align}
where we define 
  \begin{equation}
   \mathcal{K}_1^S
   \coloneqq
   \int_{\Omf} K_1^S \cdot Sv
   + \int_{\Gac} K_2^S\cdot Sv
   + \int_{\Gab} K_3^S\cdot Sv
   + \int_{\Ome} (S\Delta w - \Delta Sw) \cdot Sw_t
   .
   \label{EQ82}
  \end{equation}
In~\eqref{EQ81}, we note the cancellation of the higher-order boundary integrals due to~\eqref{EQ74} and~\eqref{EQ76}.

Next, we perform the equipartition estimates. 
We test~\eqref{EQ79} with $Sw$ and integrate by parts to obtain
  \begin{align}
   \begin{split}
    &\frac{\d}{\d t} \Big(
        \frac{\alpha}{2} \|Sw\|^2
        + \int_{\Ome} S\partial_t w_i Sw_i
        \Big)
    + \|\nabla Sw\|^2
    \\&\indeq\indeq=
    \|Sw_t\|^2 
    + \int_{\Gac} N_j \partial_j Sw_i Sw_i
    + \int_{\Ome} (S\Delta w - \Delta Sw)\cdot Sw
    .
   \end{split}
   \label{EQ83}
  \end{align}
Testing~\eqref{EQ72} with $\phi$, which is defined in~\eqref{EQ84}, we get 
  \begin{equation}
    \int_{\Omf} S\partial_t v_i \phi_i
    - \int_{\Omf} \partial_j S(a_{j\ell} a_{k\ell} \partial_k v_i) \phi_i
    + \int_{\Omf} \partial_k S(a_{ki} q) \phi_i
    = \int_{\Omf} K_1^S\cdot \phi
    .
    \label{EQ87}
  \end{equation}
We then integrate by parts and use the product rule in the time-derivative to infer 
  \begin{align}
   \begin{split}
    &\frac{\d}{\d t} \int_{\Omf} Sv_i \phi_i
    + \int_{\Omf} S(a_{j\ell}a_{k\ell} \partial_k v_i)\partial_j \phi_i
    \\&\indeq\indeq\indeq=
    \|Sv\|^2 
    + \int_{\Omf} S(a_{ki} q) \partial_k \phi_i
    + \int_{\partial\Omf} N_jS(a_{j\ell} a_{k \ell} \partial_k v_i) \phi_i
    - \int_{\partial\Omf} N_kS(a_{ki} q) \phi_i
    + \int_{\Omf} K_1^S\cdot \phi
    .
   \end{split}
   \label{EQ88}
  \end{align}
Considering the second term on the left-hand side of~\eqref{EQ88}, we expand 
  \begin{equation}
   \begin{split}
    \int_{\Omf} S(a_{j\ell}a_{k\ell} \partial_k v_i) \partial_j \phi_i
    &=
    \int_{\Omf} \nabla Sv : \nabla \phi
    + \int_{\Omf} S((a_{j\ell}a_{k\ell} - \delta_{jk})\partial_k v_i) \partial_j \phi_i
    + \int_{\Omf} (S\partial_j v_i - \partial_j Sv_i) \partial_j \phi_i
    \\&=
    \int_{\Omf} \nabla S \partial_t (\eta - \eta(\tau)) : \nabla S(\eta - \eta(\tau))
    + \int_{\Omf} \nabla S v : \nabla \widetilde{S(w(\tau))}
    \\&\indeq
    + \int_{\Omf} S((a_{j\ell}a_{k\ell} - \delta_{jk}) \partial_k v_i) \partial_j \phi_i
    + \int_{\Omf} (S\partial_j v_i - \partial_j S v_i)\partial_j \phi_i
    \\&=
    \frac{1}{2} \frac{\d}{\d t} \|\nabla S(\eta - \eta(\tau))\|^2
    + \int_{\Omf} \nabla Sv : \nabla  \widetilde{S(w(\tau))}
    \\&\indeq
    + \int_{\Omf} S((a_{j\ell}a_{k\ell} - \delta_{jk} )\partial_j v_i) \partial_j \phi_i
    + \int_{\Omf}(S \partial_j v_i - \partial_j Sv_i) \partial_j \phi_i
    ,
   \end{split}
   \label{EQ89}
  \end{equation}
where $A:B = A_{ij} B_{ij}$.
Similarly,  
  \begin{equation}
   \int_{\Omf} S(a_{ki} q)\partial_k \phi_i
   = 
   \int_{\Omf} Sq \div{\phi}
   - \int_{\Omf} S((\delta_{ki} - a_{ki})q) \partial_k \phi_i
   .
   \label{EQ90}
  \end{equation}
Incorporating~\eqref{EQ89} and~\eqref{EQ90} into~\eqref{EQ88} yields 
  \begin{align}
   \begin{split}
    &\frac{\d}{\d t} \int_{\Omf} Sv_i \phi_i
    + \frac{1}{2} \frac{\d}{\d t} \|\nabla S (\eta - \eta(\tau))\|^2
    \\&\indeq\indeq=
    \|Sv\|^2
    - \int_{\Omf} \nabla Sv : \nabla \widetilde{S(w(\tau))} 
    + \int_{\Omf} Sq \div{\phi}
    + \int_{\Omf} S((\delta_{jk} - a_{j\ell} a_{k\ell})\partial_k v_i) \partial_j \phi_i
    \\&\indeq\indeq\indeq
    - \int_{\Omf} S((\delta_{ki} -a_{ki})q)\partial_k \phi_i
    - \int_{\Omf} (S\partial_j v_i - \partial_j Sv_i)\partial_j \phi_i
    + \int_{\Omf} K_1^S \cdot \phi
    \\&\indeq\indeq\indeq
    + \int_{\partial \Omf} N_jS(a_{j\ell} a_{k\ell} \partial_kv_i)\phi_i
    - \int_{\partial\Omf} N_k S(a_{ki} q) \phi_i
    .
   \end{split}
   \label{EQ91}
  \end{align}
In particular, we examine the last two higher-order boundary terms in~\eqref{EQ91}.
Our boundary conditions~\eqref{EQ09} and~\eqref{EQ10} give us
  \begin{align}
   \begin{split}
    &\int_{\partial \Omf} S(a_{j\ell}a_{k\ell} \partial_k v_i)\Nf_j \phi_i
    - \int_{\partial \Omf}S(a_{ki}q)\Nf_k \phi_i
    \\&\indeq\indeq=
    \int_{\Gac}\Nf_j \partial_j Sw_i \phi_i
    + \int_{\Gac}(S(\Nf_j \partial_j w_i) - \Nf_j \partial_j Sw_i)\phi_i
    \\&\indeq\indeq\indeq
    - \int_{\Gab \cup \Gac} (S(\Nf_j a_{j\ell} a_{k\ell} \partial_k v_i) - \Nf_j S(a_{j\ell} a_{k\ell} \partial_k v_i))\phi_i
    + \int_{\Gab \cup \Gac}(S(\Nf_k a_{ki} q) - \Nf_k S(a_{ki} q))\phi_i
    \\&\indeq\indeq=
    -\int_{\Gac} \Ne_j \partial_j Sw_i Sw_i 
    + \mathcal{K}_2^S
    ,
   \end{split}
   \label{EQ92}
  \end{align}
where we used~\eqref{EQ86} in the last line and $\mathcal{K}_2^S$ is defined in the obvious way.
Substituting this into~\eqref{EQ91} and adding the resulting equation to~\eqref{EQ83}, the high-order boundary integrals on $\Gac$ cancel, and we recover
  \begin{align}
   \begin{split}
    &\frac{\d}{\d t} \Bigg( 
            \frac{\alpha}{2} \|Sw\|^2 
            + \int_{\Ome} Sw_t \cdot Sw
            + \frac{1}{2} \|\nabla S (\eta - \eta(\tau))\|^2
            + \int_{\Omf} Sv \cdot \phi
    \Bigg)
    + \|\nabla Sw\|^2
    \\&\indeq\indeq=
    \|Sw_t\|^2
    + \|Sv\|^2
    -\int_{\Omf} \nabla Sv \colon \nabla \widetilde{S(w(\tau))}
    + \int_{\Omf} Sq \div{\phi}
    \\&\indeq\indeq\indeq
    + \int_{\Omf} S((\delta_{jk} - a_{j\ell} a_{k\ell}) \partial_k v_i) \partial_j \phi_i
    - \int_{\Omf} S((\delta_{ki} - a_{ki})q) \partial_k \phi_i
    \\&\indeq\indeq\indeq
    - \int_{\Omf} (S\partial_j v_i - \partial_j Sv_i) \partial_j \phi_i
    + \int_{\Omf} K_1^S \cdot \phi
    + \int_{\Ome} (S\Delta w - \Delta Sw)\cdot Sw
    + \mathcal{K}_2^S
    .
   \end{split}
   \label{EQ93}
  \end{align}
Let $\lambda \in (0,1]$. 
We then multiply~\eqref{EQ93} by $\lambda$ and add it to the higher-order energy identity~\eqref{EQ81}.
From the resulting equality, we get a lower bound for the left-hand side by the Poincar\'e inequalities $\|Sv\|^2 \les \|\nabla Sv\|^2$ and $\|Sw\|^2 \les \|\nabla Sw\|^2$.
The resulting inequality is of the desired form. 
\end{proof}

\section{The total energy estimate}\label{sec_total}

As mentioned in Step~2 of the sketch (Section~\ref{sec_sketch}), in this section we show  that, if $\gamma >0$ is sufficiently small, then
\eqnb\label{est_apr}
   Y (t)
  + \lambda \int_\tau^t Y 
  \leq
  C(1 + \lambda^\alpha (t-\tau))Y(\tau)
  + C(\lambda^\alpha  + \lambda^{\kappa} (t-\tau )^2  ) \int_\tau^t Y
  + h O(Y)
\eqne
for all $0\leq \tau \leq t$ such that
\[
h(t) \leq \gamma,
\]
where $Y,h,O(Y)$ are defined in step 2 of the sketch (Section~\ref{sec_sketch}), i.e.,
\eqnb
 Y
    \coloneqq 
     \sum_{S }
      \lambda^{-c_S} \left(
       \|Sv\|^2
       + \|Sw_t\|^2
       + \|Sw\|_{1}^2
      \right),
\eqne
$c_S \geq 0$, for $S \in \{ \id , T,\partial_t, T\partial_t, T^2, \partial_{tt}\}$, $h(t) \coloneqq \sup_{[0,t)} Y^{1/2} + \int_0^t Y^{1/2}$, and $O(Y)$ denotes 
any term of the form of sums of terms involving any constant, any powers of $\lambda$ and $(t-\tau)$ and at least two factors of the form $Y^{1/2}(\tau )$, $Y^{1/2} (t)$, and~$\int_\tau^t Y^{1/2}$.\\

Given the coefficients $c_S$, which we determine below, we set
  \eqnb\label{def_abc}
   \begin{split}
    \alpha
    &\coloneqq
    \min
    \bigg\{
     2,
     2-2c_{T\partial_t} + c_{\id} + c_{\partial_{tt}},
     1-c_{\id} + c_{T\partial_t},
     1 - c_{\id} + c_{\partial_t},
     2- c_{T^2} + \min\{c_{\partial_t}, c_{\id}, c_{T^2}\},
     \\&\hspace{2cm}
     2 - c_{T} + c_{\id},
     2-c_{T^2} + \min\{c_{\partial_t}, c_{\id}, c_{T}\}, 
     2-2c_{T^2} +  c_T + \min\{c_{\partial_t}, c_{\id}, c_{T^2}\} 
     \\&\hspace{2cm}
     2-2c_T + c_{\id} + \min\{c_{\partial_t}, c_{\id}, c_T\}
    \bigg\},
    \\
    \kappa
    &\coloneqq
    \min
    \bigg\{
     2-2c_{T^2} + \min\{c_{\partial_t}, c_{\id}, c_{T^2}\} + \min\{c_{\partial_t},c_{T\partial_t}\},
     2-2c_T + \min\{c_{\partial_t}, c_{\id}, c_{T}\} + \min\{c_{\partial_t}, c_{T\partial_t}\},
  \\&\hspace{2cm}
     2-2c_T + c_{\id} + \min\{c_{\partial_t}, c_{T\partial_t}\},
     2-2c_{T^2} + \min\{c_{\partial_t}, c_{\id}, c_{T}\},  
    \bigg\}.
   \end{split}
  \eqne
We note that the above definitions of $\alpha, \gamma$ express the interdependence of the energy estimates for all\\
$S\in   \{ \id , T,\partial_t, T\partial_t, T^2, \partial_{tt}\}$ between each other, which exposes the complexity of the fluid structure interaction on a curved domain. We also set
 \eqnb\label{epsilon}
  \begin{split}
   \epsilon
   &\coloneqq
   \min
   \bigg\{
    2-2c_{T\partial_t} + c_{\id} +  \min \{ c_{\p_{tt}} , c_{\p_t}, c_{T\p_t } \} \colb ,
    2 - c_{T\partial_t} + c_{\id},
    2 - c_{T\partial_t} + c_{\partial_t} + \min\{c_{\partial_t} - 1, c_{T\partial_t} -1\},
    \\&\hspace{2cm}
    - 2c_{T^2} + c_T + \min\{c_{\partial_t}, c_{\partial_{tt}}, c_{\id}, c_{T^2}\},
    -2c_T + c_{\id} + \min\{c_{\partial_t}, c_{\partial_{tt}}, c_{\id}, c_T\},
    \\&\hspace{2cm}
    -2c_{T\partial_t} + c_{\partial_t}+ \min\{c_{\partial_{tt}},c_{\partial_t}, c_{T\partial_t}\},
   -2c_{T^2} + \min\{c_{\partial_t},c_{\id},c_{T^2}\}+ \min\{c_{\partial_t}, c_{\partial_{tt}}, c_{\id} , c_T\},\\
    &\hspace{2cm}
    -1 - 2c_{T} + c_{\id} + \min\{c_{\partial_t}, c_{T\partial_t}\},
    -2 -2c_T + c_{\partial_t} + \min\{c_{\partial_t}, c_{\partial_{tt}}, c_{\id}, c_T\},
    \\&\hspace{2cm}
    -2 - 2c_{T^2} + c_{T\partial_t} + \min\{c_{\partial_t}, c_{\partial_{tt}}, c_{\id}, c_{T^2}\},
    -2-2c_{T\partial_t} + c_{\partial_{tt}} + \min\{c_{\partial_{tt}}, c_{\partial_t}, c_{T\partial_t}\}
   \bigg\}
  \end{split}
  \eqne
In order for the energy estimate \eqref{est_apr} to yield global well-posedness and exponential decay as $t\to \infty$, we require that  
\eqnb\label{need1}
 \alpha >1, \,\,\,  \kappa >3 ,  \eqne
 as well as
 \eqnb\label{need3}
 \epsilon >0 \quad \text{ and } \quad  c_{\p_t } > c_{T\p_t } -1.
 \eqne

As mentioned in Step~4 in the sketch (Section~\ref{sec_sketch}), 
 the question of global well-posedness and exponential decay reduces to the problem of determining whether or not there exist coefficients $c_S\geq 0$, for all $S$ such that \eqref{need1} and \eqref{need3} hold. The answer to this is affirmative, and we make the choice as mentioned in \eqref{the_choice_intro}, i.e.
  \[
  \begin{aligned}
    c_{\id} &\coloneqq  5, & c_{\p_t} &\coloneqq  20/3,\\
    c_{T} &\coloneqq  5/3, & c_{T\p_t} &\coloneqq  17/3,\\
    c_{T^2} &\coloneqq  0, & c_{\p_{tt}}&\coloneqq  25/3,
  \end{aligned}
  \]
which gives that  $\alpha = 5/3$, $\kappa = 11/3 $, $\epsilon = 2/3$, so that \eqref{need1}--\eqref{need3} hold. \vspace{1cm}\\

Having fixed the coefficients $c_S$, we can derive \eqref{est_apr}, by multiplying~\eqref{EQ94} by $\lambda^{-c_S}$, integrating in time over $(\tau,t)$, summing over $S$, and using the equivalent form~\eqref{EQ95a} of $E_S$, to get 
  \begin{align}
   \begin{split}
    &
    Y (t)
    + \sum_{S} 
   \lambda^{-c_S} \left( \lambda \int_{\Omf} \left( Sv(t) -  S(v(\tau )) \right) \cdot  \widetilde{S(w (\tau ))} +  \int_\tau^t  \left(
        \|Sv\|_{1}^2
        + \|Sw_t\|^2
        + \lambda \|Sw\|_{1}^2
        \right) \right) \\
        & \hspace{6cm} + \sum_{S\in \{ \p_t , T\p_t , \p_{tt} \} } \lambda^{1-c_S } \left( \| \nabla  \eta  (t) \|^2 - \| \nabla  S (\eta  (\tau ))  \|^2 \right) 
    \\&\hspace{6cm}
    \les Y (\tau)
        + \sum_{S }
     \lambda^{-c_S} \int_\tau^t (
      L_S
      + N_S
      + C_S
      )
    ,
   \end{split}
   \label{EQ205}
  \end{align}
where we noted that $\| \nabla S (\eta (t) - \eta (\tau )) \| \geq 0$ for all $S$ not involving a time derivative, so that those terms  do not appear in the second line of~\eqref{EQ205}. In fact, the remaining terms in that line can also be absorbed:\\

\noindent If $S=\p_t$ then  
  \eqnb\label{EQ207}\begin{split}
   &\lambda^{1-c_{\p_t}}  (
       \|\nabla v(t)\|^2
       - \|\nabla v(\tau)\|^2
       )
   =
    \lambda^{1-c_{\p_t}}\int_\tau^t \frac{\d}{\d s} \|\nabla v(s)\|^2\, \d s
   \leq
   C \lambda^{1-c_{\p_t}} \int_\tau^t \|\nabla v_t\|\, \|\nabla v\|
   \\&\indeq
   \leq \frac{\lambda^{-c_{\id }}}4   \int_\tau^t\|\nabla v\|^2
   + C \lambda^{2-2c_{\p_t} + c_{\id }} \int_\tau^t\|\nabla v_t\|^2
   \leq \frac{\lambda^{-c_{\id }}}4 \int_\tau^t\|\nabla v\|^2
   + C \frac{\lambda^{-c_{\p_t} }}4 \int_\tau^t\|\nabla v_t\|^2
\end{split}  \eqne
if $\lambda >0$ is sufficiently small.\\

\noindent If $S=T\p_t $, then
\eqnb\label{EQ209}\begin{split}
  &
   \lambda^{1-c_{T\p_t}}  (
       \|\nabla T v(t)\|^2
       - \|\nabla T v(\tau)\|^2
       )
   =
    \lambda^{1-c_{T \p_t}}\int_\tau^t \frac{\d}{\d s} \|\nabla T v(s)\|^2\, \d s
   \leq
   C \lambda^{1-c_{T \p_t}} \int_\tau^t \|\nabla T v_t\|\, \|\nabla  T v\|
   \\&\indeq
   \leq \frac{\lambda^{-c_{T\p_t  }}}4   \int_\tau^t\|\nabla T v_t \|^2
   + C \lambda^{2-c_{T \p_t}} \int_\tau^t\|  v \|_2^2
   \\&\indeq
   \leq \frac{\lambda^{-c_{T\p_t  }}}4   \int_\tau^t\|\nabla T v_t \|^2
   + C \lambda^{2-c_{T \p_t}} \int_\tau^t \left( \|  v_t \|^2 + \| w_t \|_1^2 + \| \nabla T w_t \|^2 \right)
   \\&\indeq
   \leq \frac{\lambda^{-c_{T\p_t  }}}4   \int_\tau^t\|\nabla T v_t \|^2
   +  C  \int_\tau^t \left(\frac{\lambda^{-c_{\p_t}}}4 \|  v_t \|^2 +\frac{\lambda^{1-c_{\p_t}}}4 \| w_t \|_1^2 + \frac{\lambda^{1-c_{T \p_t}}}4\| \nabla T w_t \|^2 \right) 
\end{split}  \eqne
for all sufficiently small $\lambda>0$, where we used \eqref{EQ34} in the third line, and the fact that $\lambda^{2-c_{T\p_t} } < \lambda^{1-c_{\p_t}}$ (a consequence of \eqref{need3}, i.e., $c_{\p_t} > c_{T\p_t } -1$ ).\\

\noindent If $S=\p_{tt} $ then
\eqnb\label{EQ211}\begin{split}
  &   \lambda^{1-c_{\p_{tt}}}  (
       \|\nabla  v_t (t)\|^2
       - \|\nabla  v_t (\tau)\|^2
       )
   =
    \lambda^{1-c_{\p_{tt}}}\int_\tau^t \frac{\d}{\d s} \|\nabla  v_t (s)\|^2\, \d s
   \leq
   C \lambda^{1-c_{\p_{tt}}} \int_\tau^t \|\nabla  v_{tt}\|\, \|\nabla   v_t \| \\&\indeq
   \leq \frac{\lambda^{-c_{\p_t  }}}4   \int_\tau^t\|\nabla  v_t \|^2
   + \lambda^{2-2c_{\p_{tt}}+c_{\p_t}} \int_\tau^t\|  \nabla v_{tt}  \|^2
   \leq \frac{\lambda^{-c_{\p_t  }}}4   \int_\tau^t\|\nabla  v_t \|^2
   + \frac{\lambda^{-c_{\p_{tt}}}}4  \int_\tau^t\|  \nabla v_{tt}  \|^2,
\end{split}  \eqne
where we used \eqref{EQ34} in the last line. Thus, since the right-hand sides of \eqref{EQ207}--\eqref{EQ211} can be absorbed by the terms under ``$\int_\tau^t$'' on the left-hand side of \eqref{EQ205}, this concludes this step.

We now show that the first term in the bracket in \eqref{EQ205} can be neglected. Indeed, for every $S$
  \begin{equation}
    \lambda^{1-c_S} \int_{\Omega_f} (Sv(t) -  S(v(\tau)) ) \cdot  \widetilde{S(w(\tau))}
   \les 
   \lambda^{1-c_S} \lambda^2 \|Sv\|^2
   + \lambda^{1-c_S} \| S(v(\tau)) \|^2 
   + \lambda^{1-c_S} \| S(w(\tau)) \|^2
   \les
   \lambda^2 Y(t) + Y(\tau), 
   \label{EQ206}
  \end{equation}
  where the first term can indeed be absorbed by the left-hand side. Thus, \eqref{EQ205} reduces to

 \eqnb
   \begin{split}
    &
    Y (t)
    + \sum_{S} 
   \lambda^{-c_S}  \int_\tau^t \underbrace{ \left(
        \|Sv\|_{1}^2
        + \|Sw_t\|^2
        + \lambda \|Sw\|_{1}^2
        \right) }_{=: \widetilde{D_S}}
    \les Y (\tau)
        + \sum_{S }
     \lambda^{-c_S} \int_\tau^t (
      L_S
      + N_S
      + C_S
      )
    .
   \end{split}
   \label{EQ205a}
  \eqne
We will show in Sections~\ref{sec_Ls}--\ref{sec_cS} below that
  \eqnb
  \begin{split}
    \sum_{S} \lambda^{-c_S}  \int_\tau^t  L_S
    &\les (\delta + C_\delta \lambda^{\epsilon} ) 
    \sum_{S} \lambda^{-c_S}  \int_{\tau}^t \widetilde{D_S} 
   + C_\delta \left( \lambda^\alpha  (t- \tau) Y(\tau)  +  \lambda^\kappa (t-\tau )^2 \int_\tau^t Y 
    \right) + h O( Y  )
    ,\\
  \sum_{S } \lambda^{-c_S} \int_\tau^t  N_S
   &\les (h + \lambda)\lambda^{-c_{\partial_{tt}}} \int_\tau^t \|\nabla v_{tt}\|^2
   + h O(Y)
   ,\\
 \sum_{S } \lambda^{-c_S} \int_\tau^t  C_{S}
    &\les (\delta + C_\delta \lambda^\epsilon) \sum_{S} \lambda^{-c_S} \int_{\tau}^t  \widetilde{D_S}
    + C_\delta 
    \Bigg(
      \lambda^\alpha (t-\tau)Y(\tau)
      + \lambda^\kappa (t-\tau)^2 \int_\tau^t Y
    \Bigg)
    + h O(Y)
    ,
   \end{split}
   \label{EQ103}
\eqne
for all $\delta >0 $. This concludes the proof of \eqref{est_apr}, since the first terms on the right-hand sides of \eqref{EQ103} can be absorbed by the left-hand side of \eqref{EQ205a} provided \eqref{need3} holds and $\gamma , \delta , \lambda$ are taken sufficiently small. In particular,  we now fix $\gamma>0$.\\

Before we begin proving \eqref{EQ103}, we emphasize that, since $c_S\geq 0$ for all $S$, we have 
\eqnb\label{star}
g(t)\leq h(t), \qquad X(t)\leq Y(t)
\eqne
for all $t$, so that we can readily make use of all estimates from Section~\ref{sec_prelim}.

\subsection{The linear terms $L^S$}\label{sec_Ls}
Here we show the first property in \eqref{EQ103}, namely that 
\[
    \sum_{S} \lambda^{-c_S}  \int_\tau^t  L_S
    \les (\delta + C_\delta \lambda^{\epsilon} ) 
    \sum_{S} \lambda^{-c_S}  \int_{\tau}^t \widetilde{D_S} 
   + C_\delta \left( \lambda^\alpha  (t- \tau) Y(\tau)    + \lambda^\kappa (t-\tau )^2 \int_\tau^t Y \right) + hO( Y  )
   \]
for every  $\delta >0 $.
 
 Indeed, recalling~\eqref{EQ95}, we first estimate
  \begin{equation}
   L_{S,1}(t,\tau)
   = 
   -\lambda \int_{\Omf} \nabla Sv : \nabla  \widetilde{S(w(\tau))} 
   .
   \label{EQ104}
  \end{equation}
For $S\in \{ \p_t , T\p_t , \p_{tt} \}$ this term vanishes, while for $S\in \{ \id , T, T^2 \}$ we have 
\begin{align}
   \begin{split}
    \lambda^{-c_{S }}  L_{S,1} 
    \leq C 
    \lambda^{1-c_{S }} \|\nabla S v\| \|\nabla S w(\tau)\|
    &\leq 
   \delta  {\lambda^{-c_{S }}} \|\nabla S v\|^2 
    +C_\delta \lambda^{2-c_{S }} \|\nabla S w(\tau)\|^2
    \\&\leq 
\delta    {\lambda^{-c_{S }}} \|\nabla S v\|^2
    + C_\delta \lambda^2  Y (\tau),
   \end{split}
   \label{EQ105}
  \end{align}
  as required. Next, we consider
  \begin{equation}
   L_{S,2}(t,\tau)
   = 
   \lambda \int_{\Omf} Sq : \div{\phi}
   ,
   \label{EQ106a}
  \end{equation}
which requires estimates for all~$S$.
First, note that $\div{v} = (\delta_{jk} - a_{jk})\partial_j v_k$ by the  divergence-free condition~\eqref{EQ06}. Hence,
  \begin{align}
   \begin{split}
    \lambda^{-c_{\id}}  L_{\id,2}
    &\les
    \lambda^{1-c_{\id}}  \|q\| \left( \int_\tau^t \|\div{v}\|
    +  \|\nabla w(\tau)\| \right)
    \\&\les
    \lambda^{1-c_{\id}}  \|q\| \left( \int_\tau^t \|I-a \|_2 \| \nabla v\|
    +  \|\nabla w(\tau)\| \right)
    \\&\les 
    hO(Y) 
    + \lambda^{1-c_{\id}} \|q\| \|\nabla w (\tau)\|
    ,
   \end{split}
   \label{EQ107}
  \end{align}
where we used $\|I-a\|_{2} \les h(t)$. 

For the second term on the far right, we use the $H^2$ estimate~\eqref{EQ34} and \eqref{EQ35a} to get
  \begin{align}
   \begin{split}
    \lambda^{1-c_{\id}} \|q\| \,\|\nabla w(\tau)\|
    &\les \lambda^{1-c_{\id}} \left( \| v_t \| + \| w_t \|_1 +\| T w_t \|_1 \right) \|\nabla w(\tau)\|
    \\
    &\leq    \delta \left(
 \lambda^{-c_{\p_t}} \|v_t\|^2
      + \lambda^{1-c_{\p_t}}\|w_t\|_{1}^2
      + \lambda^{1-c_{T\p_t}} \|\nabla T\partial_t w\|^2
      \right)
    \\
    &\hspace{3cm}+ C_\delta \left( \lambda^{1-2c_{\id } + c_{\p_t } } + \lambda^{1-2c_{\id } + c_{T\p_t }}   \right) \|\nabla w(\tau)\|^2\\
    &\leq \delta \sum_S\lambda^{-c_S} \widetilde{D_S} + C_\delta  \lambda^\alpha Y(\tau )    
   \end{split}
   \label{EQ108}
  \end{align}
for all $\delta >0$.
\begin{remark}\label{Rem1}
We note that, for estimating $\| q \|$, we could have used the bound
  \begin{equation}
   \|q\|
   \les
   \|v\|_{2}
   + \|\nabla q\|
   \les 
   \|v_t\|
   + \|v\|_{1}
   + \|\nabla Tv\|
   ,
   \label{EQ109}
  \end{equation}
  which is more fundamental in the sense that it arises from the $H^2$ estimate \eqref{EQ33} (and \eqref{EQ35a}) for the velocity field $v$ only, without making use of the elastic structure displacement $w$, as in~\eqref{EQ34}. However, such estimate would be insufficient, as multiplying by $\lambda^{1-c_{\id}}$ makes it too big to estimate using $v$ only.
  \end{remark}

For $S = T^2$, we have 
  \begin{align}
   \begin{split}
    L_{T^2,2}
    & = \lambda  \int_{\Omf} T^2 q \div \, T^2 \phi \\
    &\les 
    \lambda \|T^2 q\| \int_\tau^t \|T^2 \div{v}\|
    + \lambda \|T^2q\| \int_\tau^t \|\div{T^2v} - T^2\div{v}\|
    + \lambda \|T^2 q\|\, \|\nabla T^2 w(\tau )\|
    \\&\les
    hO(Y)
    + \lambda\|q\|_{2} \left( \int_\tau^t\|v\|_{2}+ \|\nabla T^2 w(\tau)\| \right)
    .
   \end{split}
   \label{EQ112}
  \end{align}
The first term following the first inequality was bounded by $hO(Y)$ since the divergence-free condition \eqref{EQ06} gives
  \begin{equation}
   \|\div  v\|_{2} = \| (\delta_{ij} - a_{ij} )\p_i v_j \|_2
   \les  \|(I-a)\nabla v\|_{2} 
   \les 
   \|I-a\|_{2}\|v\|_{3}
   .
   \label{EQ112a}
  \end{equation}
 We can thus use  $H^2$ and $H^3$ Dirichlet estimates~\eqref{EQ34} and~\eqref{EQ48} to estimate $\| v \|_2$ and $\| q\|_2$, respectively, to obtain 
  \begin{align}
   \begin{split}
    &\lambda^{-c_{T^2}} L_{T^2,2}  
    \leq  h O(Y) +  \lambda^{1-c_{T^2}} (
        \|v_t\|_{1}
        + \|v\|_{1}
        + \|\nabla T^2 v\|
        ) \left( \int_\tau^t ( \|v_t\|
           + \|\nabla w_t\|
           + \|\nabla Tw_t\|
        ) + \| \nabla T^2 w (\tau ) \| \right) 
     \\&\indeq
        \leq h O(Y) + \delta \left(\lambda^{-c_{\p_t}} \| v_t \|_1^2 + \lambda^{-c_{\id}} \| v \|_1^2 + \lambda^{-c_{T^2}} \| \nabla T^2 v \|_1^2  \right)
	\\&\indeq\indeq
        + C_\delta \lambda^{2-2c_{T^2}+  \min \{ {c_{\p_t} } , {c_{\id }} , {c_{T^2}} \} }  \left( (t-\tau ) \int_\tau^t \left( \| v_t \|^2 + \| \nabla w_t \|^2 + \| \nabla T w_t \|^2 \right) + \| \nabla T^2 w (\tau ) \|^2  \right)
	\\&\indeq
        \leq h O(Y) + \delta  \sum_S \lambda^{-c_S} \widetilde{D_S}
    \\&\indeq\indeq
        + C_\delta \lambda^{2-2c_{T^2}+  \min \{ {c_{\p_t} } , {c_{\id }} , {c_{T^2}} \}  }\left( \lambda^{\min \{ c_{\p_t}c_{T\p_t} \} }   (t-\tau ) \int_\tau^t Y + \lambda^{c_{T^2}} Y(\tau )\right) 
   \end{split}
   \label{EQ113}
  \end{align}
  for all $\delta >0$.  As for $L_{T,2}$, we have analogously (using \eqref{EQ34} and \eqref{EQ35a} to estimate $\| q\|_1$) 
 \begin{align}
   \begin{split}
    \lambda^{-c_{T}} L_{T,2}  
    &\leq h O(Y) +  \lambda^{1-c_T} \| q \|_1 \left( \int_\tau^t \| v \|_1  + \| \nabla T w (\tau ) \|  \right)  \\
    &\leq h O(Y) +  \lambda^{1-c_{T}} (
        \|v_t\|
        + \|\nabla w_t \|
        + \|\nabla T w_t\|
        ) \left( \int_\tau^t \| v \|_1 + \| \nabla T w (\tau ) \| \right)
   \\&\leq
   hO(Y)
   + \delta \big(
     \lambda^{-c_{\partial_t}} \|v_t\|^2
     + \lambda^{1-c_{\partial_t}} \|\nabla w_t\|^2
     + \lambda^{1-c_{T\partial_t}}\|\nabla T w_t\|^2
     \big)
   \\&\indeq
     + C_\delta\lambda^{1-2c_{T} + \min\{c_{\partial_t}, c_{T\partial_t}\}} 
     \Bigg(
       (t-\tau) \int_\tau^t (\|v_t\|^2 + \|\nabla w_t\|^2 + \|\nabla T w_t\|^2) 
       + \|\nabla T w(\tau)\|^2
     \Bigg)
   \\&\leq
    \delta \sum_{S} \lambda^{-c_S}  \widetilde{D_S}
    + C_\delta \lambda^{1-2c_T + 2\min\{c_{\partial_t}, c_{T\partial_t}\}}(t-\tau) \int_\tau^t Y
    + C_\delta \lambda^{1-c_T + \min\{c_{\partial_t}, c_{T\partial_t}\}} Y(\tau)
   \end{split}
   \label{EQ113_T}
  \end{align}
for all $\delta > 0$.

For $S = T\p_t$, we obtain
  \begin{align}
   \begin{split}
    \lambda^{-c_{T\p_t }}  L_{T\partial_t, 2}
    &=
    \lambda^{1-c_{T\p_t } } \int_{\Omf} Tq_t : \div {Tv}
    \\&\les
    \lambda^{1-c_{T\p_t }} \| q_t \|_1  \|T \div{v}\| 
    + \lambda^{1-c_{T\p_t}} \|q_t\|_{1} \|T\div{v} - \div{Tv}\|
    \\&\les 
   \delta \lambda^{-c_{\id}} \| v \|_1^2 + C_\delta  \lambda^{2-2c_{T\p_t} + c_{\id} }\| q_t \|_1^2 + h O(Y) \\&\les 
   \delta  \lambda^{-c_{\id}} \| v \|_1^2 + C_\delta  \lambda^{2-2c_{T\p_t} + c_{\id} }\left( \|v_{tt} \|^2 + \| v_t \|_1^2 + \| \nabla T v_t \|^2 \right) + h O(Y) \\&\les 
   (\delta  + C_\delta \lambda^\epsilon ) \sum_{S} \lambda^{-c_{S}} \widetilde{D_S} + C_\delta  \lambda^{2-2c_{T\p_t} + c_{\id} }\|v_{tt} \|^2 + h O(Y) \\&\leq
   (\delta  + C_\delta \lambda^\epsilon ) \sum_{S} \lambda^{-c_{S}} \widetilde{D_S} + C_\delta  \lambda^{\alpha } Y + h O(Y)
   \end{split}
   \label{EQ110}
  \end{align}
  for all $\delta >0$,   where we used the  $H^2$ Dirichlet estimate~\eqref{EQ67a} in the fourth line, and the definition \eqref{epsilon} of $\epsilon$ in the fifth line.
  
For $S = \partial_{tt}$ we note that
  \begin{align}
    \begin{split}
   \frac1\lambda \int_\tau^t L_{\partial_{tt},2}
    &= 
     \int_\tau^t \int_{\Omf} q_{tt} \div{v_{t}}
       =
     \Bigg[\int_{\Omf} q_t \partial_t((\delta_{jk} - a_{jk})\partial_j v_k)\Bigg]_\tau^t
    - \int_\tau^t \int_{\Omf} q_t \div{v_{tt}}
    .
  \end{split}
   \label{EQ113a}
  \end{align}
Using~\eqref{EQ102b}$_1$, the first term on the right-hand side of \eqref{EQ113a} is bounded by~$hO(Y )$. For the second term, we use~\eqref{EQ102b}$_{2}$ to get
  \begin{align}
   \begin{split}
    \int_\tau^t \int_{\Omf} q_t \div{v_{tt}}
    &=
     \int_\tau^t q_t \partial_{tt}((\delta_{jk} - a_{jk}) \partial_j v_k)
    \les 
     \int_\tau^t \|q_{t}\|\, \|\partial_{tt}((I-a)\nabla  v)\| 
    \\&\les
     \int_\tau^t \|q_t\|\, \|I-a\|_{2}\|\nabla v_{tt}\|
    + hO(Y)
    .
   \end{split}
   \label{EQ113b}
  \end{align}
  Thus
  \eqnb
  \begin{split}
  \lambda^{-c_{\p_{tt}}} \int_\tau^t L_{\p_{tt},2} &\les \delta \lambda^{- c_{\p_{tt}}} \int_\tau^t \| \nabla v_{tt} \|^2  + C_\delta  \lambda^{2-c_{\p_{tt}}} \int_\tau^t \|q_t\|^2 \|I-a\|_2^2
    + h O(Y)\\
    &\les \delta  \lambda^{- c_{\p_{tt}}} \int_\tau^t \| \nabla v_{tt} \|^2  + C_\delta h O(Y)
    \end{split}
  \eqne
  for all $\delta >0$. 

Finally, since $\|\div \, v \| \les h O(Y^{1/2})$, we obtain 
  \begin{equation}
   L_{\partial_t,2}
   \les
   h O(Y), 
   \label{EQ114}
  \end{equation}
using~\eqref{EQ102a}--\eqref{EQ102b}. 
\subsection{The nonlinear terms $N_S$}\label{sec_Ns}
Here we show the second inequality in \eqref{EQ103}, namely that 
  \[
   \sum_{S } \lambda^{-c_S}\int_\tau^t N_S
   \les
   (h + \lambda)\lambda^{-c_{\partial_{tt}}} \int_\tau^t \|\nabla v_{tt}\|^2
   + hO(Y)
   .
   \]
Indeed, we first note that for all cases other than $S = \partial_{tt}$, we have 
  \begin{equation}
   \rho_S N_S(t,\tau)
   \les
   hO(Y)
   ,
   \label{EQ116}
  \end{equation}
by the estimates~\eqref{EQ102a}--\eqref{EQ102b}, \eqref{star} and H\"older's inequality. 

For $S = \partial_{tt}$, we get for the first term in $N_S$,
  \begin{equation}
   \lambda^{-c_{\partial_{tt}}}\int_{\Omf} (\delta_{kj} - a_{j\ell}a_{k\ell})\partial_k Sv_i \partial_j Sv_i
    \les 
    \lambda^{-c_{\partial_{tt}}}\|I-aa^T\|_2\|\nabla v_{tt}\|^2
    \les 
    h \lambda^{-c_{\partial_{tt}}} \|\nabla v_{tt}\|^2
    .
    \label{EQ117}
  \end{equation}
For the latter two terms in $N_S$, we integrate by parts in time to obtain 
  \begin{align}
   \begin{split}
    &
    \lambda^{1-c_{\partial_{tt}}} \int_\tau^t \int_{\Omf} \partial_{tt}((\delta_{jk} - a_{j\ell}a_{k\ell})\partial_k v_i)  \partial_j \phi_i
    - \lambda^{1-c_{\partial_{tt}}}\int_\tau^t\int_{\Omf} \partial_{tt}((\delta_{ki} - a_{ki})q) \partial_k \phi_i
    \\&\indeq\indeq \les
    \lambda^{1-c_{\partial_{tt}}} \int_\tau^t \|\nabla v_{tt}\|(
        \|\partial_t ((I-aa^T)\nabla v)\|
        + \|\partial_t ((I-a)q)\|
        )
    \\&\indeq\indeq\indeq 
    + \lambda^{1-c_{\partial_{tt}}} \Bigg[
        \int_{\Omf} \partial_t ((\delta_{jk} - a_{j\ell} a_{k\ell})\partial_k v_i) \partial_j \partial_t v_i
        -\int_{\Omf} \partial_t((\delta_{ki} - a_{ki})q)\partial_k \partial_t v_i
        \Bigg]_\tau^t
    \\&\indeq\indeq\les
    \lambda^{1-c_{\partial_{tt}}} \int_\tau^t \|\nabla v_{tt} \|^2
    + hO(Y)
    ,
   \end{split}
   \label{EQ118}
  \end{align} 
where we also used \eqref{EQ102a}--\eqref{EQ102b} and~\eqref{star}. This completes the estimate for the nonlinear term. 

\subsection{The commutator terms $C_S$}\label{sec_cS}
Here we prove the last inequality in \eqref{EQ103}, namely that 
  \begin{equation}
    \sum_{S } 
    \lambda^{-c_S} \int_\tau^t C_{S,1}
    \les 
    (\delta + C_\delta \lambda^\epsilon) \sum_{S} \lambda^{-c_S} \int_{\tau}^t  \widetilde{D_S}
    + C_\delta 
    \Bigg(
      \lambda^\kappa (t-\tau)^2 \int_\tau^t Y
      + \lambda^\alpha (t-\tau)Y(\tau)
    \Bigg)
    + hO(Y)
   \label{EQ118a}
  \end{equation}
and
   \begin{equation}
    \sum_{S}\lambda^{-c_S} \int_\tau^t C_{S,2}
    \les 
    \big(
      \delta 
      + C_\delta \lambda^\epsilon
    \big)
    \sum_{S} \lambda^{-c_S} \int_\tau^t \widetilde{D_S}
    + hO(Y)
   \label{EQ130}
  \end{equation}
  for all $\delta> 0$.\vspace{1cm}\\
  
 In order to establish \eqref{EQ118a}, we  only consider $S \in \{T,T^2,T\partial_t\}$, as otherwise $C_{S,1}=0$.
We start with the first term in the definition~\eqref{EQ98} of $C_{S,1}$.
For $S = T$, we have 
  \begin{align}
   \begin{split}
    &
    -\lambda^{1-c_{T}} \int_{\Omf} (T\partial_j v_i - \partial_j Tv_i)\partial_j \phi_i
    \\&\indeq\indeq\les
    \lambda^{1-c_T}\|T\nabla v - \nabla T v\|
    \Bigg(
      \int_\tau^t \|\nabla Tv\|
      + \|\nabla Tw(\tau)\|
    \Bigg)
    \\&\indeq\indeq\les
    \lambda^{1-c_T}\|v\|_1 
    \Bigg(
      \int_{\tau}^t\|v\|_2
      + \|\nabla Tw(\tau)\|
    \Bigg)
    \\&\indeq\indeq\les
    \delta \lambda^{-c_{\id}}\|v\|_1^2
    + C_\delta\lambda^{2-2c_T + c_{\id}} 
      \Bigg(
       (t-\tau)\int_\tau^t \big(
        \|v_t\|^2
        + \|w_t\|_1^2
        + \|\nabla Tw_t\|^2
       \big)
       + \|\nabla Tw(\tau)\|^2
      \Bigg)
    \\&\indeq\indeq\les
    \delta \lambda^{-c_{\id}} \|v\|_1^2
    + C_\delta \lambda^{2-2c_T +c_{\id}}
      \Bigg(
       \lambda^{\min\{c_{\partial_t}, c_{T\partial_t}\}} (t-\tau)\int_\tau^t Y
       + \lambda^{c_T} Y(\tau)
      \Bigg)
    \\&\indeq\indeq\les
    \delta \sum_{S} \lambda^{-c_S} \widetilde{D_S} 
    + C_\delta \lambda^\kappa (t-\tau) \int_\tau^t Y
    + C_\delta \lambda^{\alpha} Y(\tau)
   \end{split}
   \label{EQ119}
  \end{align}
for all $\delta >0$, where we used \eqref{EQ34} in the third inequality.
For $S = T^2$, we obtain 
  \begin{align}
   \begin{split}
    &
    -\lambda^{1-c_{T^2}}  \int_{\Omf} (T^2\partial_j v_i - \partial_j T^2v_i)\partial_j \phi_i
    \\&\indeq\indeq\les
    \lambda^{1-c_{T^2}} \|v\|_2
    \Bigg(
    \int_\tau^t \|\nabla T^2v\|
    + \|\nabla T^2w(\tau)\|
    \Bigg)
    \\&\indeq\indeq\les
    \lambda^{1-c_{T^2}} 
    \big(
      \|v_t\| 
      + \|v\|_{1}
      + \|\nabla Tv\|
    \big)
    \Bigg(
    \int_\tau^t \|v\|_3 
    + \|\nabla T^2 w(\tau)\|
    \Bigg)
    \\&\indeq\indeq\les
    \delta \big(
      \lambda^{-c_{\partial_t}}\|v_t\|^2
      + \lambda^{-c_{\id}} \|v\|_{1}^2
      + \lambda^{-c_T} \|\nabla Tv\|^2
    \big)
    \\&\indeq\indeq\indeq
    + C_\delta\lambda^{2- 2c_{T^2} + \min\{c_{\partial_t}, c_{\id},c_{T}\}}
    \Bigg(
      (t-\tau) \int_\tau^t 
      \|v\|_3 
      + \|\nabla T^2 w\|
    \Bigg)
    \\&\indeq\indeq\les
    \delta \sum_S \lambda^{-c_S} \widetilde{D_S}
    + C_\delta \lambda^{ 2- 2c_{T^2} + \min\{c_{\partial_t}, c_{\id},c_{T}\}}
    \Bigg(
      (t-\tau)
      \int_\tau^t Y
      + \lambda^{c_{T^2}}Y(\tau)
    \Bigg)
    \\&\indeq\indeq\les
    \delta \sum_S \lambda^{-c_S} \widetilde{D_S}
    + C_\delta
    \Bigg(
      \lambda^\kappa (t-\tau)
      \int_\tau^t Y
      + \lambda^\alpha Y(\tau)
    \Bigg)
   \end{split}
   \label{EQ120}
  \end{align}
for all $\delta>0$, where we used \eqref{EQ33} in the second inequality.  Now, let $S = T\partial_t$.
We obtain
  \begin{align}
   \begin{split}
    &
    \lambda^{1-c_{T\partial_t}} \int_{\Omf} (T\partial_j \partial_t v_i - \partial_j T \partial_t v_i)\partial_j Tv_i
    \les
    \lambda^{1-c_{T\partial_t}} \|T\nabla v_t - \nabla T v_t\|_{L^2}\|v\|_{2}
    \\&\indeq\indeq\les 
    \delta\lambda^{-c_{\partial_t}} \|v_t\|_{1}^2
    + C_\delta\lambda^{2-2c_{T\partial_t} + c_{\partial_t}}
    \big(
      \|v_t\|^2
      + \|\nabla w_t\|^2
      + \|\nabla T w_t\|^2
    \big)
    \\&\indeq\indeq\les
    \big(
      \delta 
      + C_\delta \lambda^{ 2-c_{T\partial_t} + c_{\partial_t} + \min\{ c_{\partial_t}-1, c_{T\partial_t} -1\}}
    \big)
    \sum_{S} \lambda^{-c_S} \widetilde{D_S}
    \\&\indeq\indeq\les
    \big(
      \delta
      + C_\delta \lambda^\epsilon
    \big)
    \sum_S \lambda^{-c_S} \widetilde{D_S}
   \end{split}
   \label{EQ121}
  \end{align}
for all $\delta>0$, where we used \eqref{EQ34} in the second inequality. Next, we treat the second and third terms in tandem. 
For $S = T^2$, we get
  \begin{align}
   \begin{split}
    &
    \lambda^{1-c_{T^2}} \int_{\Omf} (T^2 \partial_j (a_{j\ell} a_{k\ell} \partial_k v_i) - \partial_j T^2 (a_{j\ell}a_{k\ell} \partial_k v_i)) \phi_i
    - \lambda^{1-c_{T^2}} \int_{\Omf} (T^2 \partial_k(a_{ki}q) - \partial_k T^2(a_{ki}q))\phi_i
    \\&\indeq\indeq\les
    \lambda^{1-c_{T^2}}
    \big(
      \|aa^T\nabla v\|_{2}
      + \|aq\|_{2}
    \big)
    \Bigg(
      \int_\tau^t \|T^2 v\|
      + \|T^2 w(\tau)\|
    \Bigg)
    \\&\indeq\indeq\les
    \lambda^{1-c_{T^2}} 
    \big(
      \|v\|_{3}
      + \|q\|_{2}
    \big)
    \Bigg(
      \int_\tau^t \|v\|_2
      + \|\nabla T w(\tau)\|
    \Bigg) + h O(Y)
    \\&\indeq\indeq\les
    \delta\big(
      \lambda^{-c_{\partial_t}}\|v_t\|_{1}^2
      + \lambda^{-c_{\id}}\|v\|_{1}^2
      + \lambda^{-c_{T^2}}\|\nabla T^2 v\|^2
    \big)
    \\&\indeq\indeq\indeq
    + C_\delta\lambda^{2-2c_{T^2} +\min\{c_{\partial_t}, c_{\id}, c_{T^2}\}}
    \Bigg(
      (t-\tau) \int_\tau^t 
      (
        \|v_t\|
        + \|w_t\|_1
        + \|Tw_t\|_1
      )
      + \|\nabla T w(\tau)\|^{2}
    \Bigg) + h O(Y)
    \\&\indeq\indeq\les
    \delta \sum_S \lambda^{-c_S} \widetilde{D_S}
    + C_\delta\lambda^{2-2c_{T^2} +\min\{c_{\partial_t}, c_{\id}, c_{T^2}\}}
    \Bigg(
      \lambda^{\min\{c_{\partial_t}, c_{T\partial_t}\}}(t-\tau) \int_\tau^t Y
      + \lambda^{c_T}Y(\tau)
    \Bigg)+ h O(Y)
    \\&\indeq\indeq\les
    \delta \sum_S \lambda^{-c_S} \widetilde{D_S}
    + C_\delta 
    \Bigg(
      \lambda^\kappa (t-\tau) \int_\tau^t Y
      + \lambda^\alpha Y(\tau)
    \Bigg)+ h O(Y)
   \end{split}
   \label{EQ122}
  \end{align}
for all $\delta>0$, where we used \eqref{i-a} in the second inequality, and both~\eqref{EQ34} and~\eqref{EQ48} in the third inequality. 

In a similar fashion, we get for $S = T$, 
  \begin{align}
   \begin{split}
    &
    \lambda^{1-c_{T}} \int_{\Omf} (T \partial_j (a_{j\ell} a_{k\ell} \partial_k v_i) - \partial_j T (a_{j\ell}a_{k\ell} \partial_k v_i)) \phi_i
    - \lambda^{1-c_{T}} \int_{\Omf} (T \partial_k(a_{ki}q) - \partial_k T(a_{ki}q))\phi_i
    \\&\indeq\indeq\les
    \delta \sum_{S} \lambda^{-c_S} \widetilde{D_S}
    + C_\delta \lambda^{2-2c_T + \min\{c_{\partial_t}, c_{\id}, c_T\}}
    \Bigg(
      \lambda^{\min\{c_{\partial_t},c_{T\partial_t}\}} (t-\tau) \int_\tau^t Y
      + \lambda^{c_{\id}} Y(\tau)
    \Bigg),
   \end{split}
   \label{EQ122b}
  \end{align} 
for all $\delta > 0$. 
For $S = T\partial_t$, we note that 
  \begin{equation}
   \|\partial_t (aa^T \nabla v)\|_{1}
   \les 
   \|\partial_t(aa^T) \nabla v\|_{1}
   + \|aa^T \nabla v_t\|_{1}
   \les
   hY^{1/2}
   + \|v_t\|_{2}
   \label{EQ122a}
  \end{equation} 
  and similarly $\| \p_t (aq ) \|_1 \les hY^{1/2} + \| q_t \|_1$, by \eqref{i-a}--\eqref{EQ102b}. Hence, 
  \begin{align}
   \begin{split}
    &
    \lambda^{1-c_{T\partial_t}} \int_{\Omf}(T\partial_j\partial_t(a_{j\ell} a_{k\ell} \partial_k v_i) - \partial_j T\partial_t(a_{j\ell} a_{k\ell} \partial_k v_i)) \phi_i
    - \lambda^{1-c_{T\partial_t}} \int_{\Omf} (T \partial_k\partial_t (a_{ki} q) - \partial_k T\partial_t(a_{ki} q)) \phi_i
    \\&\indeq\indeq\les 
    \lambda^{1-c_{T\partial_t}}
    \big(
        \|v_t\|_{2}
        + \|q_t\|_{1}
    \big)
    \|v\|_{1}
    + hO(Y)
    \\&\indeq\indeq\les 
    \delta \lambda^{-c_{\id}}\|v\|_{1}^2
    + C_\delta\lambda^{2-2c_{T\partial_t} + c_{\id}}
    \big(
      \|v_{tt}\|^2
      + \|v_{t}\|_{1}^2
      + \|\nabla T v_t\|^2
    \big)
    + hO(Y)
    \\&\indeq\indeq\les
    \big(
      \delta 
      + C_\delta \lambda^\epsilon
    \big) 
    \sum_S \lambda^{-c_S} \widetilde{D_S}
    + hO(Y)
   \end{split}
   \label{EQ123}
  \end{align} 
  for all $\delta>0$, where we used \eqref{EQ67a} in the second inequality. 
Next, we analyze the fourth and fifth terms in~$C_{S,1}$. 
Recall that~$\phi|_{\Gamma_c} = Sw(t)|_{\Gamma_c}$.
Then for $S = T^2$, we obtain through Lemma~\ref{L01} 
  \begin{align}
   \begin{split}
    &
    \lambda^{1-c_{T^2}} \int_{\Ome} (T^2 \Delta w - \Delta T^2 w) \cdot T^2 w
    - \lambda^{1-c_{T^2}} \int_{\Gac} (T^2(N_j \partial_j w) - N_j \partial_j T^2 w) \cdot T^2 w
    \\&\indeq\indeq\les
    \lambda^{1-c_{T^2}} \|w\|_{3} \|\nabla Tw\|
    + \lambda^{1-c_{T^2}} \|T^2 (N_j \partial_j w) - N_j \partial_j T^2 w\|_{H^{1/2}(\Gamma_c)} \|Tw\|_{H^{1/2}(\Gamma_c)}
    \\&\indeq\indeq\les
    \lambda^{1-c_{T^2}} \|w\|_{3}\|\nabla T w\|
    \\&\indeq\indeq\les
    \delta\big(
      \lambda^{1-c_{\partial_t}}\|w_t\|_{1}^2
      + \lambda^{1-c_{\partial_{tt}}}\|w_{tt}\|_{1}^2
      + \lambda^{1-c_{\id}}\|\nabla w\|^2
      + \lambda^{1-c_{T^2}}\|\nabla T^2 w\|^2
        \big) 
    \\&\indeq\indeq\indeq
    + C_\delta\lambda^{1 - 2c_{T^2} + \min\{c_{\partial_t},c_{\partial_{tt}},c_{\id},c_{T^2}\}}
      \|\nabla Tw\|^2
    \\&\indeq\indeq\les
    \big(
      \delta 
      + C_\delta \lambda^{-2 c_{T^2} + c_T + \min\{c_{\partial_t},c_{\partial_{tt}},c_{\id},c_{T^2}\}}
    \big)
    \sum_S  \lambda^{-c_S} \widetilde{D_S}
    \\&\indeq\indeq\les
    \big(
      \delta
      + C_\delta \lambda^\epsilon
    \big)
    \sum_S  \lambda^{-c_S} \widetilde{D_S}
   \end{split}
  \label{EQ124}
  \end{align}
for all $\delta >0$, where we used \eqref{EQ44} in the second inequality.
For $S=T$, we perform an analogous computation using lower order estimates to obtain
  \begin{align}
   \begin{split}
    &
    \lambda^{1-c_{T}} \int_{\Ome} (T \Delta w - \Delta T w) \cdot T w
    - \lambda^{1-c_{T}} \int_{\Gac} (T(N_j \partial_j w) - N_j \partial_j T w) \cdot T w
    \\&\indeq\indeq\les
    \big(
     \delta
     + C_\delta \lambda^{-2c_T + c_{\id} + \min\{c_{\partial_t}, c_{\partial_{tt}}, c_{\id}, c_T\}}
    \big)
    \sum_S  \lambda^{-c_S} \widetilde{D_S}
   \end{split}
   \label{EQ124a}
  \end{align}
for all $\delta > 0$. 
For $S = T\partial_t$, we have through Lemma~\ref{L01} 
  \begin{align}
   \begin{split}
    &
    \lambda^{1-c_{T\partial_t}}  \int_{\Ome} (T\Delta w_t - \Delta T w_t)\cdot Tw_t
    - \lambda^{1-c_{T\partial_t}} \int_{\Gac} (T(N_j \partial_j \partial_t w_i) - N_j \partial_j T \partial_tw_i) T \partial_t w_i
    \\&\indeq\indeq\les 
    \lambda^{1-c_{T\partial_t}} \|w_t\|_{2}\|w_t\|_{1}
    \\&\indeq\indeq\les
    \delta\big(
      \lambda^{-c_{\partial_{tt}}}\|\partial_{tt}w_t\|^2
      + \lambda^{1-c_{\partial_{tt}}}\|w_{tt}\|^2
      + \lambda^{1-c_{\partial_t}}\|w_t\|_{1}^2
      + \lambda^{1-c_{T\partial_t}}\|Tw_t\|_{1}^2
      \big)
    \\&\indeq\indeq\indeq
    + C_\delta\lambda^{ 1 -2 c_{T\p_t } + \min\{c_{\partial_{tt}},c_{\partial_t}, c_{T\partial_t}\}}
    \|w_t\|_{1}^2
    \\&\indeq\indeq\les
    \big(
      \delta 
      + C_\delta \lambda^{-2c_{T\partial_t} + c_{\partial_t}+ \min\{c_{\partial_{tt}},c_{\partial_t}, c_{T\partial_t}\}}
    \big) 
    \sum_S \lambda^{-c_S} \widetilde{D_S}
    \\&\indeq\indeq\les
    \big(
      \delta
      + C_\delta \lambda^\epsilon
    \big)
    \sum_S \lambda^{-c_S} \widetilde{D_S}
   \end{split}
   \label{EQ125}
  \end{align}
for all $\delta >0$, where, in the second inequality, we used the first inequality in~\eqref{EQ45}. 
Finally, we estimate the last two boundary integrals.
For $S = T^2$, we obtain
  \begin{align}
   \begin{split}
    &
    \lambda^{1-c_{T^2}} \int_{\Gab \cup \Gac} (T^2(N_j^fa_{j\ell}a_{k\ell} \partial_k v_i) - N_j^fT^2(a_{j\ell} a_{k\ell} \partial_k v_i))\phi_i
    -\lambda^{1-c_{T^2}} \int_{\Gab \cup \Gac} (T^2(N_k^f a_{ki}q) - N_k^fT^2(a_{ki}q))\phi_i
    \\&\indeq\indeq\les
    \lambda^{1-c_{T^2}} 
    \big(
        \|v\|_{H^{5/2}(\Gab \cup \Gac)}
        + \|q\|_{H^{3/2}(\Gab\cup \Gac)}
    \big)
    \Bigg(
      \int_\tau^t  \|T v\|_{1}
      +  \|T w(\tau)\|_1 
    \Bigg)
    \\&\indeq\indeq\les
    \lambda^{1-c_{T^2}}
    \big(
      \|v\|_{3}^2
      + \|q\|_{2}^2
    \big)
    \Bigg(
      \int_\tau^t \|\nabla Tv\|
      + \|\nabla T w(\tau)\|
    \Bigg)
    ,
   \end{split}
   \llabel{EQ126}
  \end{align}
from which the estimate follows exactly as in~\eqref{EQ122}.
Note the use of Lemma~\ref{L01} in the first inequality.
The estimates for $S = T\partial_t$ and $S = T$ are also identical to that of the second and third terms of $C_{S,1}$ after using Lemma~\ref{L01}. 
This completes the proof of the lemma.\\

We now verify~\eqref{EQ130}. To this end, we first note that  all terms vanish for $S = \id$. 
We start with the first line in~\eqref{EQ98}.
In the case when $S = T^2$ for the first term, three derivatives fall on~$a$. 
Since we have no control on $\|a\|_{3}$, we integrate by parts. 
We first see that
  \begin{align}
   \begin{split}
    &
    \lambda^{-c_{T^2}}\int_{\Omf} (T^2 \partial_j(a_{j\ell}a_{k\ell} \partial_k v_i) - \partial_j(a_{j\ell}a_{k\ell} \partial_k T^2 v_i))T^2 v_i
    \\&\indeq =
    \lambda^{-c_{T^2}}\int_{\Omf} (T^2 \partial_j(a_{j\ell}a_{k\ell} \partial_k v_i) - \partial_j T^2(a_{j\ell}a_{k\ell} \partial_k v_i))T^2 v_i
    \\&\indeq\indeq
    + \lambda^{-c_{T^2}}\int_{\Omf} \partial_j (T^2 (a_{j\ell} a_{k\ell} \partial_k v_i)- a_{j\ell}a_{k\ell} \partial_k T^2 v_i) T^2 v_i
    \\&\indeq = 
    I_1 + I_2
    .
   \end{split}
   \label{EQ131}
  \end{align}
Note that $I_1 \les \|v\|_{3}\|\nabla Tv\| + hO(Y)$. 
Integrating by parts in $I_2$ gives us 
  \begin{align}
   \begin{split}
    I_2
    &=
    -\lambda^{-c_{T^2}}\int_{\Omf} (T^2(a_{j\ell}a_{k\ell} \partial_k v_i) - a_{j\ell} a_{k\ell} \partial_k T^2 v_i) \partial_j T^2 v
    \\&\indeq
    + \lambda^{-c_{T^2}}\int_{\partial \Omf} N_j (T^2(a_{j\ell}a_{k\ell} \partial_k v_i) - a_{j\ell} a_{k\ell} \partial_k T^2 v_i)T^2 v_i
    \\&\les
     \lambda^{-c_{T^2}}\|v\|_{3}\|v\|_{2}
     + hO(Y)
    ,
   \end{split}
   \label{EQ132}
  \end{align}
where we used Lemma~\ref{L01} for the boundary integral.
Also note that the commutators in~\eqref{EQ132} give extra terms of the form $T^2(aa^T)\nabla v$ and $T(aa^T)T\nabla v$.
Indeed, these terms are bounded by~$hO(Y)$.
For example,
  \begin{equation}
   \|T^2 (aa^T)\nabla v\|\|\nabla T^2 v\|
   \les
   \|T^2\big(aa^T\big)\| \|v\|_{3}^2
   \les
   hO(Y)
   .
   \label{EQ132a}
  \end{equation}
Without integrating by parts, we find for the pressure term in the first line of the definition of $C_{S,2}$ in~\eqref{EQ98},
  \begin{equation}
   -\lambda^{-c_{T^2}}\int_{\Omf} (T^2(a_{ki} \partial_k q) - a_{ki} \partial_kT^2q)T^2 v_i
   \les
   \lambda^{-c_{T^2}}\|q\|_{2}\|v\|_{2}
   + hO(Y)
   .
   \label{EQ133}
  \end{equation}
Finally, note that
  \begin{align}
    \begin{split}
    &
   \lambda^{-c_{T^2}}(\|v\|_{3} + \|q\|_{2}) \|v\|_{2}
  \\&\indeq\indeq\les
   \delta\big(
       \lambda^{-c_{\partial_t}}\|v_t\|_{1}^2
       + \lambda^{-c_{\id}}\|v\|_{1}^2
       + \lambda^{-c_{T^2}}\|\nabla T^2 v\|^2
       \big)
   \\&\indeq\indeq\indeq
   + C_\delta \lambda^{-2c_{T^2} + \min\{c_{\partial_t},c_{\id},c_{T^2}\}}
       \big(
       \|v_t\|^2
       + \|v\|_{1}^2
       + \|\nabla Tv\|^2
       \big)
   \\&\indeq\indeq\les
   \big(\delta +C_\delta \lambda^{-2c_{T^2} + \min\{c_{\partial_t},c_{\id},c_{T^2}\}+ \min\{c_{\partial_t}, c_{\partial_{tt}}, c_T\}}\big) \sum_S \lambda^{-c_S} \widetilde{D_S}
   \\&\indeq\indeq\les
   \big(
     \delta
     + C_\delta \lambda^\epsilon
   \big)
   \sum_S \lambda^{-c_S} \widetilde{D_S}
  \end{split}
   \label{EQ134}
  \end{align}
for all $\delta >0$, where we used \eqref{EQ48} in the first inequality.
In an analogous manner, we get for $S = T$, 
  \begin{align}
   \begin{split}
    &
    \lambda^{-c_T} \int_{\Omf} (T \partial_j (a_{j\ell} a_{k\ell} \partial_k v_i) - \partial_j (a_{j\ell} a_{k\ell} \partial_k Tv_i)) Tv_i
    -\lambda^{-c_T} \int_{\Omf} (T(a_{ki} \partial_k q) - a_{ki} \partial_k T q) Tv_i
    \\&\indeq\indeq \les
    \big(
      \delta 
      + C_\delta \lambda^{-2c_T + c_{\id} + \min\{c_{\partial_t}, c_{\id}, c_{T}\}}
    \big)
    \sum_S \widetilde{D_S} + h O(Y)
   \end{split}
   \label{EQ134a}
  \end{align}
for all $\delta > 0$.
Next, we consider $S = T\partial_t$. 
We get from~\eqref{EQ67},
  \begin{align}
   \begin{split}
    &
    \lambda^{-c_{T\partial_t}} \int_{\Omf}(T\partial_t \partial_j (a_{j\ell} a_{k\ell} \partial_k v_i) - \partial_j (a_{j\ell} a_{k\ell} \partial_k T\partial_tv_i))T\partial_t v_i
    -\lambda^{-c_{T\partial_t}} \int_{\Omf} (T\partial_t(a_{ki} \partial_k q) - a_{ki} \partial_k Tq_t) T\partial_t v_i
    \\&\indeq\indeq\les
    \lambda^{-c_{T\partial_t}} (\|v_t\|_{2} + \|q_t\|_{1})\|v_t\|_{1}
    + hO(Y)
    \\&\indeq\indeq\les
    \delta\big(
        \lambda^{-c_{\partial_{tt}}}\|v_{tt}\|^2
        + \lambda^{-c_{\partial_t}}\|v_t\|_{1}^2
        + \lambda^{-c_{T\partial_t}}\|\nabla Tv_t\|^2
        \big) 
    \\&\indeq\indeq\indeq
    + C_\delta\lambda^{-2c_{T\partial_t} + \min\{c_{\partial_{tt}}, c_{\partial_t}, c_{T\partial_t}\}}\|v_t\|_{1}^2
    + hO(Y)
    \\&\indeq\indeq\les
    \big(\delta + C_\delta \lambda^{ -2c_{T\partial_t} + c_{\partial_t} + \min\{c_{\partial_{tt}}, c_{\partial_t}, c_{T\partial_t}\}}\big) \sum_S \lambda^{-c_S} \widetilde{D_S}
    + hO(Y)
    \\&\indeq\indeq\les
    \big(
      \delta
      + C_\delta \lambda^\epsilon
    \big)
    \sum_S \lambda^{-c_S} \widetilde{D_S}
    + hO(Y)
   \end{split}
   \llabel{EQ135}
  \end{align}
for all $\delta >0$, where we  used \eqref{EQ67a} in the second inequality.
For $S = \partial_{tt}$, the highest-order terms within the commutators cancel. 
Indeed, expanding the time derivatives, we obtain 
  \begin{align}
   \begin{split}
    &
    \lambda^{-c_{\partial_{tt}}}\int_{\Omf} (\partial_{tt} \partial_j (a_{j\ell}a_{k\ell} \partial_k v_i) -\partial_j (a_{j\ell}a_{k\ell} \partial_k \partial_{tt}v_i)) \partial_{tt} v_i
    + \lambda^{-c_{\partial_{tt}}}\int_{\Omf} (\partial_{tt}(a_{ki} \partial_k q) - a_{ki} \partial_k q_{tt}) \partial_{tt} v_i
    \\&\indeq=
    \lambda^{-c_{\partial_{tt}}}\int_{\Omf} (\partial_j(\partial_{tt}(a_{j\ell}a_{k\ell}) \partial_k v_i) + \partial_{tt}a_{ki} q) \partial_{tt} v_i
    + 2\lambda^{-c_{\partial_{tt}}}\int_{\Omf} (\partial_j(\partial_t (a_{j\ell}a_{k\ell}) \partial_k \partial_tv_i) + \partial_t a_{ki} q_t)\partial_{tt}v_i
    \\&\indeq \les
    \lambda^{-c_{\partial_{tt}}} 
    \big(
      \|\partial_{tt}(aa^T)\nabla v\|_{1} 
      + \|a_{tt} q\|\|v_{tt}\|
    \big)
    + \lambda^{-c_{\partial_{tt}}} 
    \big(
      \|\partial_t (aa^T)\nabla v_t\|_{1} 
      + \|a_tq_t\|
    \big) \|v_{tt}\|
    \\&\indeq\les
    hO(Y)
    ,
   \end{split}
   \llabel{EQ136}
  \end{align}
since $\|a_{tt}\|_{1}, \|a_{t}\|_{2} \les h$.
The case when  $S = \partial_t$ is similar. 

The second line in the definition of $C_{S,2}$ in~\eqref{EQ98} is handled identically to the first line after using Lemma~\ref{L01}. 

Finally, we estimate the third line of the definition \eqref{EQ98} of~$C_{S,2}$. 
Considering only the first term, we rewrite 
  \begin{equation}
   \lambda^{-c_S}\int_{\Omf} Sq a_{ki} \partial_k Sv_i
   = 
   \lambda^{-c_S}\int_{\Omf} Sq a_{ki} (\partial_k Sv_i - S \partial_k v_i)
   + \lambda^{-c_S}\int_{\Omf} Sq (a_{k i} S \partial_k v_i  - S(a_{ki} \partial_k v_i))
   =
   I_{S,1} + I_{S,2}
   ,
   \label{EQ150}
  \end{equation}
where we used the divergence-free condition in \eqref{EQ06} in~$I_{S,2}$.
For $S = \id$, both integrals vanish. 
In fact, 
  \begin{equation}
   I_{S,2}
   \les
   hO(Y)
   ,
   \label{EQ151}
  \end{equation}
  for all $S \in \{T,\partial_t, T\partial_t, T^2\}$, since for those $S$ a derivative falls directly on $a$ (so that we can use \eqref{EQ20} or \eqref{EQ25} to obtain an extra factor of $O(Y^{1/2})$), and there are less than $2$ time derivatives falling on $q$ (so that we do not need to integrate by parts in time; and so we can use \eqref{EQ66} or \eqref{EQ67} to control $q_t$).   

To estimate $I_{\partial_{tt},2}$, we expand the commutator to obtain
  \begin{equation}
    I_{\partial_{tt},2}
    = 
    - \lambda^{-c_{\partial_{tt}}} \int_{\Omf} q_{tt} \partial_{tt} a_{ki} \partial_k v_i
    - 2\lambda^{-c_{\partial_{tt}}} \int_{\Omf} q_{tt} \partial_ta_{ki} \partial_k \partial_t v_i
    =
    I_{\partial_{tt},21} 
    + I_{\partial_{tt},22}
    .
    \label{EQ152}
  \end{equation}
We estimate only the first term, as the second is easier.
Integrating by parts in time, we get
  \begin{align}
    \int_\tau^t I_{\partial_{tt},21}
    &= 
    -\lambda^{-c_{\partial_{tt}}} \Bigg[
            \int_{\Omf} q_t \partial_{tt} a_{ki} \partial_k v_i
    \Bigg]_\tau^t
    + \lambda^{-c_{\partial_{tt}}} \int_\tau^t \int_{\Omf} q_t \partial_{ttt}a_{ki} \partial_k v_i
    + \lambda^{-c_{\partial_{tt}}}\int_\tau^t q_t \partial_{tt} a_{ki} \partial_k \partial_t v_i
    \label{EQ153}
    \\&\les
    \lambda^{-c_{\partial_{tt}}} \int_\tau^t \|q_t\|_{1} \|v\|_{2} \|a_{ttt}\|
    + hO(Y)
    ,
    \label{EQ155}
  \end{align}
where we used $\|a_{tt}\|_{2} \les h$ in the first and third terms on the right-side of~\eqref{EQ153} and the embedding $H^1 \subset L^4$ in the middle term. 
From~\eqref{EQ19}, we have 
  \begin{equation}
   \|a_{ttt}\|
   \les 
   \|\nabla v_{tt}\|\|\nabla \eta\|_{2}
   + \|v\|_{3} \|v_t\|_{1}
   ,
   \label{EQ156}
  \end{equation}
and so, since $\|\nabla \eta\|_{2} \les 1 + h$ (recall \eqref{i-a}), we find that
  \begin{equation}
  \begin{split}
   \int_\tau^t I_{\partial_{tt},21}
   &\les 
   \int_\tau^t\lambda^{-c_{\partial_{tt}}}\|q_t\|_{1} \|v\|_{2}\|\nabla v_{tt}\|
   + hO(Y)
   \\&\les
   \int_\tau^t \delta \lambda^{-c_{\partial_{tt}}}\|\nabla v_{tt}\|^2
   + C_\delta hO(Y)
   \\&\les
   \delta \sum_{S}  \int_\tau^t \lambda^{-c_S} \widetilde{D_S}
   + hO(Y)
  \end{split}
   \label{EQ157}
  \end{equation}
for all $\delta >0$. 
Next, we bound~$I_{S,1}$. 
We need only to consider $S \in \{T,T^2,T\partial_t\}$.
We see that \eqref{EQ48} and the first line of \eqref{EQ34} imply that
  \begin{align}
   \begin{split}
    I_{T^2,1}
    &\les
    \lambda^{-c_{T^2}}\|q\|_{2} \|\nabla T^2 v - T^2 \nabla v\|
    \les
    \lambda^{-c_{T^2}}\|q\|_{2}\|v\|_{2}
    \\&\les
    \delta \big(
        \lambda^{-c_{\partial_t}}\|v_t\|_{1}^2 
        + \lambda^{-c_{\id}}\|v\|_{1}^2
        + \lambda^{-c_{T^2}}\|\nabla T^2 v\|^2
        \big)
    \\&\indeq
    + C_\delta \lambda^{-2c_{T^2} + \min\{c_{\partial_t}, c_{\id},c_{T^2}\}}
    \big(
        \|v_t\|^2
        + \|v\|_{1}^2
        + \|\nabla Tv\|^2
        \big)
    \\&\les
    (\delta + C_\delta \lambda^{ -2c_{T^2} + \min\{c_{\partial_t}, c_{\id}, c_{T^2}\} + \min\{c_{\partial_t}, c_{\id}, c_{T}\}}) \sum_S \lambda^{-c_S} \widetilde{D_S}
    \\&\les
    \big(
      \delta
      + C_\delta \lambda^\epsilon
    \big)
    \sum_S \lambda^{-c_S} \widetilde{D_S}
   \end{split}
   \label{EQ158}
  \end{align}
for all $\delta >0$. 
For $S = T$, we get
\begin{align}
   \begin{split}
    I_{T,1}
    &\les
    \lambda^{-c_T} \|q\|_1 \|\nabla T v - T\nabla v\|
    \les \lambda^{-c_T} \|q\|_1\|v\|_1
    \\&\les
    \delta \big(
      \lambda^{-c_{\partial_t}} \|v_t\|^2
      + \lambda^{1-c_{\partial_t}} \|w_t\|_1^2
      + \lambda^{1-c_{T\partial_t}} \|\nabla T w_t\|^2
      \big)
    \\&\indeq
    + C_\delta \lambda^{-1-2c_T + \min\{c_{\partial_t}, c_{T\partial_t}\}} \|v\|_1^2
    \\&\les
    \big(
      \delta
      + C_\delta \lambda^{-1 -2c_T + c_{\id} + \min\{c_{\partial_t}, c_{T\partial_t}\} }
    \big)
    \sum_S \lambda^{-c_S} \widetilde{D_S}
    \\&\les
    \big(
      \delta
      + C_\delta \lambda^{\epsilon}
    \big)
    \sum_{S} \lambda^{-c_S} \widetilde{D_S}
    ,
   \end{split}
   \label{EQ158a}
  \end{align}
for all $\delta > 0$.
Moreover, we get from~\eqref{EQ67a} that 
  \begin{align}
   \begin{split}
    I_{T\partial_t,1}
    &\les
    \lambda^{-c_{T\partial_t}} \|q_t\|_{1} \|\nabla Tv_t - T\nabla Tv_t\|
    \les 
    \lambda^{-c_{T\partial_t}} \|q_t\|_{1} \|v_t\|_{1}
    \\&\les
    \delta \big(
        \lambda^{-c_{\partial_{tt}}}\|v_{tt}\|^2
        + \lambda^{-c_{\partial_t}}\|v_t\|_{1}^2
        + \lambda^{-c_{T\partial_t}}\|\nabla Tv_t\|^2
        \big)
    \\&\indeq
    + C_\delta\lambda^{-2c_{T\partial_t} + \min\{c_{\partial_{tt}}, c_{\partial_t}, c_{T\partial_t}\}}
    \|v_t\|_{1}^2
    + hO(Y)
    \\&\les
    \big(\delta + C_\delta \lambda^{-2c_{T\partial_t} + c_{\partial_t} + \min\{c_{\partial_{tt}}, c_{\partial_t}, c_{T\partial_t}\}}\big) \sum_S \lambda^{-c_S} \widetilde{D_S}
    \\&\les
    \big(
      \delta
      +C_\delta \lambda^\epsilon
    \big)
    \sum_S \lambda^{-c_S} \widetilde{D_S}
   \end{split}
   \label{EQ159}
  \end{align}
for all $\delta >0$.
Now, we handle the last two terms of the third line of $C_{S,2}$ in tandem.
Note that they vanish for $S = \partial_t$ and $S = \partial_{tt}$. 
For $S = T$, we obtain 
  \begin{align}
   \begin{split}
    &
    \lambda^{-c_T} \int_{\Omega_e} (T\Delta w - \Delta T w)\cdot Tw_t
    + \lambda^{-c_T}\int_{\Gamma_c} (T(N_j \partial_j w_i) - N_j \partial_j Tw_i) T\partial_t w_i
    \\&\indeq\indeq\les
    \lambda^{-c_T} \|w\|_{2}\|w_t\|_{1}
    \\&\indeq\indeq\les
    \delta\big( 
       \lambda^{1-c_{\partial_t}}\|w_{t}\|^2
       + \lambda^{1-c_{\partial_{tt}}}\|w_{tt}\|^2
       + \lambda^{1-c_{\id}}\|w\|_{1}^2
       + \lambda^{1-c_{T}}\|\nabla Tw\|^2
        \big)
    \\&\indeq\indeq\indeq
    + C_\delta\lambda^{-1 -2c_T + \min\{c_{\partial_t}, c_{\partial_{tt}}, c_{\id}, c_T\}}
        \|w_{t}\|_{1}^2
    \\&\indeq\indeq\les
    \big(\delta + C_\delta \lambda^{-2 -2c_T + c_{\partial_t} + \min\{c_{\partial_t}, c_{\partial_{tt}}, c_{\id}, c_T\}}\big)
    \sum_S \lambda^{-c_S}\widetilde{D_S}
    \\&\indeq\indeq\les
    \big(
      \delta 
      + C_\delta\lambda^\epsilon
    \big)
    \sum_S \lambda^{-c_S} \widetilde{D_S}
   \end{split}
   \label{EQ200}
  \end{align}
for all $\delta >0$, where we used \eqref{EQ43} in the second inequality.
For $S = T^2$, we find that 
  \begin{align}
   \begin{split}
    &
    \lambda^{-c_{T^2}}\int_{\Omega_e} (T^2\Delta w - \Delta T^2 w) \cdot T^2 w_t
    + \lambda^{-c_{T^2}}\int_{\Gamma_c} (T^2(N_j \partial_j w_i) - N_j\partial_j T^2 w_i) T^2\partial_t w_i
    \\&\indeq\indeq\les
    \lambda^{-c_{T^2}}\|w\|_{3} \|\nabla Tw_t\|
    \\&\indeq\indeq\les
    \delta \big(
        \lambda^{1-c_{\partial_t}}\|w_t\|_{1}^2
        + \lambda^{1-c_{\partial_{tt}}}\|w_{tt}\|_{1}^2
        + \lambda^{1-c_{\id}}\|w\|_{1}^2
        + \lambda^{1-c_{T^2}}\|\nabla T^2 w\|^2
        \big)
    \\&\indeq\indeq\indeq
    + C_\delta\lambda^{-1 -2c_{T^2} + \min\{c_{\partial_t}, c_{\partial_{tt}}, c_{\id}, c_{T^2}\}}
        \|\nabla T w_t\|^2
    \\&\indeq\indeq\les
    \big(\delta + C_\delta \lambda^{-2 - 2c_{T^2} + c_{T\partial_t} + \min\{c_{\partial_t}, c_{\partial_{tt}}, c_{\id}, c_{T^2}\}}\big) 
    \sum_S \lambda^{-c_S} \widetilde{D_S}
    \\&\indeq\indeq\les
    \big(
      \delta
      + C_\delta \lambda^\epsilon
    \big)
    \sum_S \lambda^{-c_S} \widetilde{D_S}
   \end{split}
   \label{EQ201}
  \end{align}
for all $\delta >0$, where we used \eqref{EQ44} in the second inequality. 
Finally, for $S = T\partial_t$, we have 
  \begin{align}
   \begin{split}
    &
    \lambda^{-c_{T\partial_t}} \int_{\Omega_e} (T\Delta w_t - \Delta Tw_t)\cdot Tw_{tt}
    + \lambda^{-c_{T\partial_t}} \int_{\Gamma_c} (T(N_j \partial_j \partial_t w_i) - N_j\partial_jT \partial_t w_i) T \partial_{tt} w_i
    \\&\indeq\indeq\les 
    \lambda^{-c_{T\partial_t}} \|w_t\|_{2} \|w_{tt}\|_{1}
    \\&\indeq\indeq\les
    \delta \big(
        \lambda^{- c_{\partial_{tt}}}\|w_{ttt}\|^2
        + \lambda^{1-c_{\partial_{tt}}}\|w_{tt}\|^2
        + \lambda^{1-c_{\partial_t}}\|w_t\|_{1}^2
        + \lambda^{1-c_{T\partial_t}}\|\nabla Tw_t\|^2
        \big)
    \\&\indeq\indeq\indeq
    + C_\delta \lambda^{-1 - 2c_{T\partial_t} + \min\{c_{\partial_{tt}}, c_{\partial_t}, c_{T\partial_t}\}} 
        \|w_{tt}\|_{1}^2
    \\&\indeq\indeq\les
    \big(\delta + C_\delta\lambda^{ -2-2c_{T\partial_t} + c_{\partial_{tt}} + \min\{c_{\partial_{tt}}, c_{\partial_t}, c_{T\partial_t}\} }\big) 
    \sum_S \lambda^{-c_S} \widetilde{D_S}
    \\&\indeq\indeq\les
    \big(
      \delta
      + C_\delta \lambda^\epsilon
    \big)
    \sum_S \lambda^{-c_S} \widetilde{D_S}
   \end{split}
   \label{EQ202}
  \end{align}
for all $\delta >0$, where, in the second inequality, we used the first inequality of~\eqref{EQ45}. This completes the proof of~\eqref{EQ130}.

\section{An ODE-type Lemma}\label{sec_ode}
In this section we provide a proof for the ODE-type lemma used in Theorem~\ref{T01}. 
The proof is essentially the same as in~\cite{KO}, however, the powers of $\lambda$ are slightly more general. 
\begin{lemma}
\label{L10}
Given $C \geq 1$, $\gamma \in (0,1]$, $\alpha >1$, $\beta > 2$, $\kappa >3$, and sufficiently small $\lambda >0$, there exists $\eps > 0$ with the following property. 
Suppose that $f \colon [0,\infty) \to [0,\infty)$ is a continuous function satisfying
  \begin{equation}
   f(t) 
   + \lambda \int_\tau^t f
   \leq
   C(1 + \lambda^{\alpha} (t-\tau))f(\tau)
   + C(\lambda^{2} + \lambda^\beta (t-\tau)+ \lambda^\kappa (t-\tau )^2) \int_\tau^t f
   + h(t) O(f)
   ,
   \label{EQ300}
  \end{equation}
for all times $t > 0$ such that 
  \begin{equation}
   h(t)
   := 
   \sup_{[0,t)} f^{1/2}
   + \int_0^t f^{1/2}
   \leq 
   \gamma
   ,
   \label{EQ301}
  \end{equation}
and all $\tau \in [0,t]$, where $O(f)$ denotes any term involving powers of $\lambda$, $C$, and $(t-\tau)$ and at least two factors of the form $f^{1/2}(\tau)$, $f^{1/2}(t)$, $\int_\tau^t f^{1/2}$. 
Then the condition $f(0) \leq \eps$ implies that 
  \begin{equation}
   f(t)
   \leq 
   A\eps \ee^{-t/a}
   ,
   \label{EQ302}
  \end{equation}
for all $t \geq 0$, where $a := 2C/\lambda$ and $A := 30C$.
\end{lemma}
Note that, as mentioned in Step~3 in Section~\ref{sec_sketch}, applying Lemma~\ref{L10} to \eqref{est_apr}, we obtain the claim of Theorem~\ref{T01} by taking $\lambda>0$ sufficiently small.

\begin{proof}[Proof of Lemma~\ref{L10}.]
First, let $\eps > 0$ be smaller than $\gamma^2 \lambda^2/10^4 C^3$. 
Then $\eps^{1/2} < \gamma$. 
Hence, we can use continuity to get that~\eqref{EQ301} holds for $t \in [0,T)$ for some $T > 0$. 
Then as long as~\eqref{EQ302} holds,~\eqref{EQ301} holds as well and implies that $h(t) < 4a(A\eps)^{1/2} < \gamma$. 
To see this, note that $\sup_{[0,t)} f^{1/2} \leq A^{1/2} \eps^{1/2} < \gamma/2$ and 
  \begin{equation}
   \int_0^t f^{1/2}
   \leq 
   \int_0^\infty f^{1/2}
   \leq 
   A^{1/2} \eps^{1/2} \int_0^\infty \ee^{-s/2a} \,ds
   = 2a A^{1/2} \eps^{1/2}
   < \frac{4C}{\lambda}(30C)^{1/2} \frac{\gamma \lambda}{10^2 C^{3/2}}
   < \frac{\gamma}{2}
   .
   \label{EQ303}
  \end{equation}
Hence, it is sufficient to verify~\eqref{EQ302} for all $t \geq 0$ assuming that~\eqref{EQ301} holds up to~$t$. 
Suppose that the claim in the lemma is false, and let $t > 0$ be the first time such that 
  \begin{equation}
   f(t)
   = 
   A\eps \ee^{-t/a}
   .
   \label{EQ304}
  \end{equation}
Then let $\tau \in [0,t)$ be the last time such that 
  \begin{equation}
   f(\tau) = 2\eps \ee^{-\tau/a}
   .
   \label{EQ305}
  \end{equation}
First, we claim that 
  \begin{equation}
   t- \tau
   \leq 
   \frac{4C}{\lambda}
   .
   \label{EQ306}
  \end{equation}
If this is false, then applying~\eqref{EQ300}  on $[\tau, \tau + 4C/\lambda]$, and  taking into account only the second term on the left-hand side, we get 
  \begin{align}
   \begin{split}
    4C \eps \ee^{-\tau/a}(1-\ee^{-2})
    &=
    2\eps \lambda \int_\tau^{\tau + 4C/\lambda}\ee^{-s/a} \, \d s
    \leq
    \lambda \int_\tau^{\tau + 4C/\lambda}f
    \\&\leq
    2C\eps (1 + 4C\lambda^{\alpha - 1})\ee^{-\tau/a}
    + C(2  C\lambda +  8  C\lambda^{\beta -2} + 32 C^2 \lambda^{\kappa -3})A \eps \ee^{-\tau/a}
    + \widetilde{C} \eps^{3/2} \ee^{-\tau /a}
    ,
    \label{EQ307}
   \end{split}
  \end{align}
where $\widetilde{C} > 0$ only depends on $C$, $\lambda$, and the form of~$hO(f)$.
For the last term in~\eqref{EQ307}, we used $h\eps \leq 4a A^{1/2}\eps^{3/2}$.
Now, note that $1-\ee^{-2} \geq 3/4$, and that $\lambda$ is sufficiently small so that~\eqref{EQ307} gives us 
  \begin{equation}
   3C
   \leq
   2C
   +8C^2 \lambda^{\alpha -1}
   + 120C^3(\lambda + \lambda^{\beta-2}+ \lambda^{\kappa-3} )
   + \widetilde{C} \eps^{1/2}
   .
   \label{EQ308}
  \end{equation}
  Thus, if $\lambda$ is sufficiently small, depending on $C,\alpha, \beta , \kappa$, we can choose $\varepsilon$ sufficiently small so that
and with our choice of $\lambda$ and with $\eps$ sufficiently small, we have
  \begin{equation}
   3C
   \leq
   2C + C/2
   ,
   \label{EQ309}
  \end{equation}
which is a contradiction. 
Hence,~\eqref{EQ306} holds.

Next, we apply~\eqref{EQ300} on $[\tau,t]$, and considering only the first term on the left-hand-side gives us 
  \begin{equation}
   30C \eps \ee^{-t/a}
   =
   f(t)
   \leq 
   2\eps C(1 + 4C\lambda^{\alpha -1})\ee^{-\tau /a}
   + 30C^2 (2  C\lambda + 8  C\lambda^{\beta -2} + 32 C^2 \lambda^{\kappa-3} ) \eps \ee^{-\tau/a}
   + \widetilde{C}\eps^{3/2}\ee^{-\tau/a}
   .
   \label{EQ310}
  \end{equation}
Note that $\ee^{-\tau/a} = \ee^{-\tau/a} \ee^{(t-\tau)/a} \leq \ee^{-t/a} \ee^2 \leq 10\ee^{-t/a}$. 
Taking this bound into account in~\eqref{EQ310} and dividing through by $\eps \ee^{-t/a}$, we get
  \begin{equation}
   30C
   \leq
   20C
   + 80C^2\lambda^{\alpha -1}
   + 1200 C^3 (\lambda + \lambda^{\beta -2} +\lambda^{\kappa-3}  )
   + 10\widetilde{C}\eps^{1/2}
   .
   \label{EQ311}
  \end{equation}
  Thus, as above, if $\lambda$ is sufficiently small, we can choose $\varepsilon$ small enough to get a contradiction.
\end{proof}

\section*{Acknowledgments}
BI and IK were supported in part by the NSF grant DMS-2205493,
while WO was supported by the NSF grant DMS-2511556 and by the Simons grant SFI-MPS-TSM-00014233.

\end{document}